\documentclass[a4paper,12pt]{amsart}

\usepackage{float}
\usepackage{euscript,eufrak,verbatim}
\usepackage{graphicx}
\usepackage{amscd,pifont}
\usepackage[usenames]{color}
\usepackage[colorlinks,linkcolor=red,anchorcolor=blue,citecolor=blue]{hyperref}
\usepackage{bbm}

\newtheorem{theorem}{Theorem}[section]
\newtheorem{lemma}[theorem]{Lemma}
\newtheorem{proposition}[theorem]{Proposition}
\newtheorem{corollary}[theorem]{Corollary}

\def\R{\mathbb{R}}
\def\N{\mathbb{N}}
\def\Z{\mathbb{Z}}
\def\Q{\mathbb{Q}}
\def\Zp{\mathbb{Z}_{p}}
\def\Qp{\mathbb{Q}_{p}}

\begin{document}

\title[Compact open spectral sets in $\mathbb{Q}_p$]{Compact open spectral sets in $\mathbb{Q}_p$ }

\date{\today}

\author%[authorlabel1]
{Aihua Fan}
\address%[authorlabel1]
{School of Mathematics and Statistics, Central  China Normal University, Wuhan, 430079,  China \& 
LAMFA, UMR 7352 CNRS, University of Picardie,
33 rue Saint Leu, 80039 Amiens, France}
\email{ai-hua.fan@u-picardie.fr}

\author%[authorlabel1]
{Shilei Fan}
\address%[authorlabel1]
{School of Mathematics and Statistics \& Hubei Key Laboratory of Mathematical Sciences, Central  China Normal University,  Wuhan, 430079,  China}
\email{slfan@mail.ccnu.edu.cn}

\author%[authorlabel1]
{Ruxi Shi}
\address%[authorlabel1]
{LAMFA, UMR 7352 CNRS, University of Picardie,
33 rue Saint Leu, 80039 Amiens \& Departement de mathematiques et application, Ecole Normale Superieure, 75005, Paris, France}
\email{ruxi.shi@u-picardie.fr}

\thanks{A. H. FAN was supported by NSF of China (Grant No. 11471132) and self-determined research funds of CCNU (Grant No. CCNU14Z01002); S. L. FAN was supported by NSF of China (Grant No.s 11401236 and 11231009)}

\begin{abstract}
In this article, we prove that a compact open set  in the field $\Qp$ of $p$-adic numbers is a spectral set if and only if it tiles $\Qp$ by translation, and also if and only if it is $p$-homogeneous which is easy to check. We also characterize spectral sets in $\mathbb{Z}/p^n \mathbb{Z}$
($p\ge 2$ prime, $n\ge 1$ integer) by  tiling property and also by homogeneity.
Moreover, we construct a class of singular spectral measures in $\Qp$, some of which are self-similar measures.
\end{abstract}


\footnote[1]{Revised version to appear in Journal of Functional Analysis}

\maketitle
\section{Introduction}\label{introduction}
The problem that we consider is generally rised for all locally compact Abelian groups  and the results that we obtain concern only the field $\Qp$ of $p$-adic numbers ($p\ge 2$ being a prime).
Let us first state the problem.
Let $G$ be a locally compact Abelian group and $\Omega\subset G$ be a Borel set of positive and finite Haar measure. The set $\Omega$ is said to be {\em spectral} if there exists a set $\Lambda \subset \widehat{G}$ of continuous characters of $G$ which forms
a Hilbert basis of the space $L^2(\Omega)$ of square Haar-integrable functions. Such a set
$\Lambda$ is called a {\em spectrum} of
$\Omega$ and $(\Omega,\Lambda)$ is called a {\em spectral pair}. We say that the set $\Omega$
{\em  tiles}  $G$ by translation if there exists a set $T \subset G$ of translates
such that $\sum_{t\in T} 1_\Omega(x-t) =1$ for almost all $x\in G$, where $1_A$ denotes the indicator function of a set $A$. Such a set $T$ is called a
{\em tiling complement} of $\Omega$ and $(\Omega, T)$ is called a {\em tiling pair}.
The so-called {\em spectral set conjecture} states that $\Omega$ is a spectral set if and only if $\Omega$ tiles $G$.

This conjecture in the case $G=\mathbb{R}^d$ is the famous Fuglede spectral set conjecture \cite{Fuglede1974}. Both the original Fuglede conjecture and the generalized conjecture stated above have
attracted considerable attention over the last decades. For the case of $\mathbb{R}^d$, many
positive results were obtained \cite{Iosevich-Katz-Tao2003,Iosevich-Pedersen1998,Jorgensen-Pedersen1992,Jorgensen-Pedersen1999,Kolountzakis2000convex,Kolountzakis2000,Lagarias-Wang1997,Lagarias-Reed-Wang2000} before Tao \cite{Tao2004}  disproved it by showing that the
direction ``Spectral $\Rightarrow$ Tiling" does not hold when $d\ge 5$.
 Now it is known that the conjecture is false in both directions for $d \geq 3$ \cite{FMM06,Kolountzakis2006,KM2006,Matolcsi2005}.   However, the conjecture is still open in lower dimensions ($d=1, 2$).  On the other hand, Iosevich, Katz and Tao   \cite{Iosevich-Katz-Tao2003} proved that  Fuglede's conjecture is true for convex planar sets. The non-convex case is considerably more complicated, and is not understood even in dimension $1$.
  Lagarias and Wang \cite{Lagarias-WangInvent,Lagarias-Wang1997} proved that all tilings of $\R$ by a bounded region must be periodic, and that the corresponding translation sets are rational up to affine transformations. This in turn leads to a structure theorem for bounded tiles, which would be crucial for the direction ``Tiling$\Rightarrow $ spectral ".
  %It was also observed in [15] that the ?tiling implies spectrum? part of Fuglede?s conjecture for compact sets in R would follow from a conjecture of Tijdeman [20] concerning factorization of finite cyclic groups; however, Tijdeman?s conjecture is now known to fail without additional assumptions (see [13] for a discussion). 
  Assume that $\Omega\subset \mathbb{R}$ is a finite union of intervals. The conjecture holds when $\Omega$ is a union of two intervals \cite{Laba2001}. If $\Omega$ is a union of three intervals, it is known that ``Tiling$\Rightarrow $ spectral "; and ``Spectral $\Rightarrow$ Tiling"holds with ``an additional hypothesis" \cite{BKKM10,BM11,BM14}.  
  
 The problem
for local fields was considered by the first author of the present paper in \cite{Fan} where among others, is proved the basic Landau
theorem concerning the Beurling density of spectrum.
In this paper, we consider the conjecture restricted  for compact open sets in the   field  $\Qp$ of $p$-adic numbers.  

We shall give a geometric characterization of compact open spectral sets and prove that a compact open set is a spectral set if and only if it tiles $\Q_p$. 
The spectra  and the tiling complements of compact open spectral sets are also investigated. Subject to an isometric transformation of $\Q_p$, the spectra and tiling complements are unique and determined by the set of possible distances of different  points in the compact open spectral set.

Actually, in \cite{FFLS2015}, we prove that the conjecture holds in $\Q_p$ without the  compact open restriction. Moreover, any spectral set is proved to be  a compact open set up to a Haar-null set.

Let us recall some notions and  notation (we refer to \cite{Vladimirov-Volovich-Zelenov1994}). The ring of $p$-adic integers is denoted by $\Zp$ and the Haar measure on $\Q_p$ is denoted by $\mathfrak{m}$ or $dx$. We assume that the Haar measure is normalized so that $\mathfrak{m}(\Zp)=1$. The dual group $\widehat{\Q}_p$ of  $\Q_p$ is isomorphic to $\Q_p$.
Any $x\in \Qp$ can be written as
$$
   x = \sum_{n=v_p(x)}^\infty a_n p^n    \quad (v_p(x) \in \mathbb{Z}, a_n \in\{0, 1, 2, \cdots, p-1\} \hbox{ and  }a_{v_p(x)}\neq 0).
$$
Here, the integer $v_p(x)$ is called the {\em $p$-valuation} of $x$.
The fractional part $\{x\}$ of $x$ is defined to be $ \sum_{n=v_p(x)}^{-1} a_n p^n$. We fix
the following  character $\chi \in \widehat{\Q}_p$:
$$
      \chi(x) = e^{2\pi i\{x\}}.
$$
Notice that $\chi$ is equal to $1$ on $\Zp$
but is non-constant on $p^{-1}\Zp$. For any $y\in \Qp$, we define
$$\chi_y(x) = \chi(yx).$$
Then the map $y \mapsto \chi_y$ from
$\Qp$ onto $\widehat{\Q}_p$ is an isomorphism.
%
%
%Let $\mu$ be a probability Borel measure on $\Qp$. We say that $\mu$ is a {\em spectral measure}
%if there exists a set $\Lambda \subset \widehat{\Qp}$ such that
%$\{\chi_\lambda\}_{\lambda \in \Lambda}$ is an orthonormal basis (i.e. a Hilbert basis) of $L^2(\mu)$. Then
%$\Lambda$ is called a {\em spectrum} of $\mu$ and we call $(\mu, \Lambda)$ a {\em spectral pair}.
%Assume that $\Omega$ is a set in $\Qp$
%of positive and finite Haar measure. That $\Omega$ is a {\em spectral set} means the restricted measure $\frac{1}{\mathfrak{m}(\Omega)}\mathfrak{m}|_\Omega$
%is a spectral measure. In this case, instead of saying $(\frac{1}{\mathfrak{m}(\Omega)}\mathfrak{m}|_\Omega, \Lambda)$ is a spectral pair,  we say that
%$(\Omega, \Lambda)$ a {\em spectral pair}.

\begin{figure}[H]
  \centering
  \includegraphics[width=0.9\textwidth]{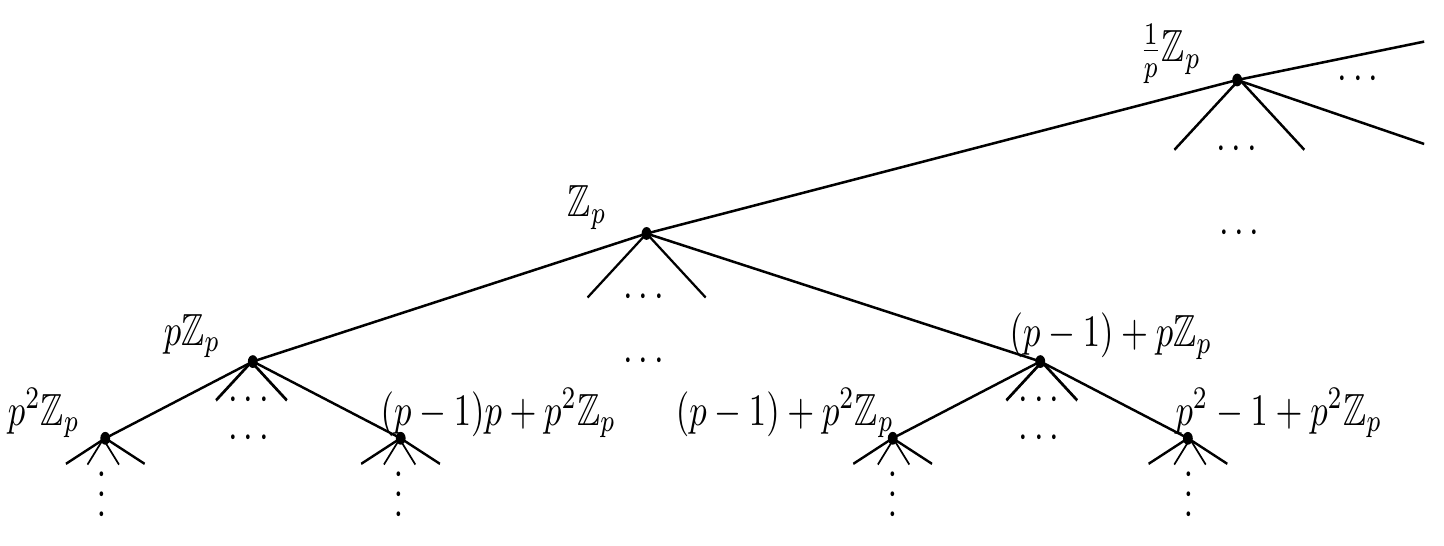}\\
  \caption{Consider $\Q_p$ as an infinite tree.}
\label{Qp}
\end{figure}

In order to state our main result, we consider the field $\mathbb{Q}_p$ as an infinite tree
$(\mathcal{T}, \mathcal{E})$. The set of vertices $\mathcal{T}$ is the set of all balls
in $\mathbb{Q}_p$. The set of edges $\mathcal{E}$ is the subset of $\mathcal{T}\times \mathcal{T}$ consisting of pairs $(B', B) \in \mathcal{T}\times \mathcal{T}$ such that
$$
    B' \subset B, \qquad |B|=p |B^\prime|,
$$
where $|B|$ denotes the Haar measure of the ball $B$. 
This fact will be denoted by $B' \prec B$. We call $B'$ a {\em descendant} of $B$, and $B$ the {\em parent} of $B'$.

%In this paper, we characterize compact open spectral sets and a  class of  singular spectral %measures on $\Qp$.
Any bounded open set $O$ of $\mathbb{Q}_p$ can be described by a subtree $(\mathcal{T}_O, \mathcal{E}_O)$ of $(\mathcal{T}, \mathcal{E})$.
In fact, let $B^*$ be the smallest ball containing $O$, which
 will be the root of the tree. For any given ball $B$ contained in $O$,
 there is a unique sequence of balls $B_0, B_1, \cdots, B_r$ such that
 $$
    B=B_0 \prec B_1 \prec B_2 \prec \cdots \prec B_r = B^*.
 $$
 We assume that the set  of vertices $\mathcal{T}_O$ is composed of all such balls
  $B_0, B_1, \cdots, B_r$ for all possible balls $B$ contained in $O$.
  The set of edges $\mathcal{E}_O$ is composed of all edges $B_i \prec B_{i+1}$
  as above.
  \begin{figure}[H]
  \centering
  \includegraphics[width=1\textwidth]{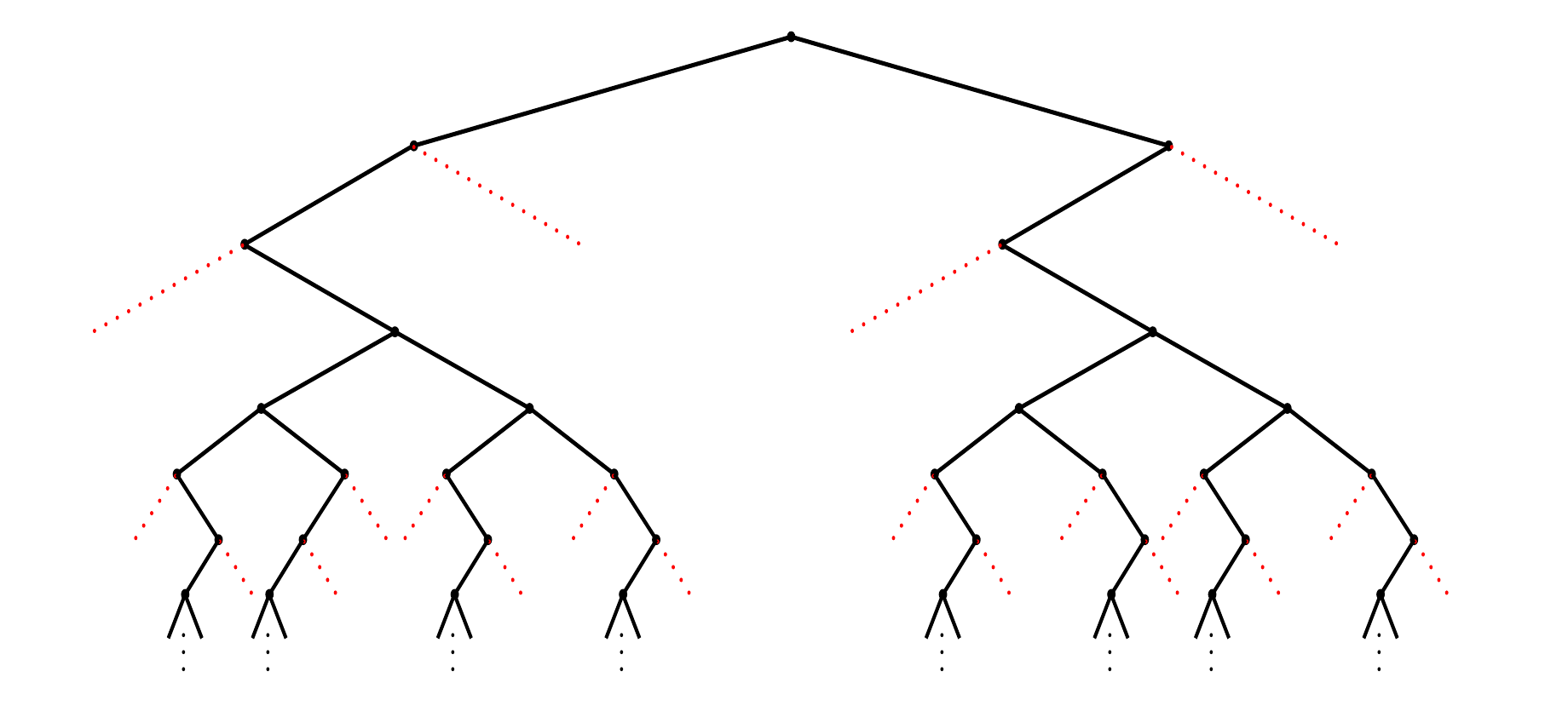}\\
  \caption{For $p=2$, a $p$-homogeneous tree.}
\label{TIJ1}
\end{figure}

A subtree $(\mathcal{T}', \mathcal{E}')$ is said to be {\em homogeneous} if the number of descendants of $B\in \mathcal{T}'$ depends only on $|B|$. If this number is either
 $1$ or $p$, we call $(\mathcal{T}', \mathcal{E}')$ a  {\em $p$-homogeneous} tree.

 A bounded open set is said to be {\em homogeneous} (resp. {\em $p$-homogeneous}) if the corresponding tree
is homogeneous (resp. $p$-homogeneous).

Any compact open set can be described by a finite tree, because a compact open set
 is a disjoint finite union of balls of same size. In this case, as in the above construction of subtree
 we only consider these balls of same size as $B$.
\begin{figure}[H]
  \centering
  \includegraphics[width=0.9\textwidth]{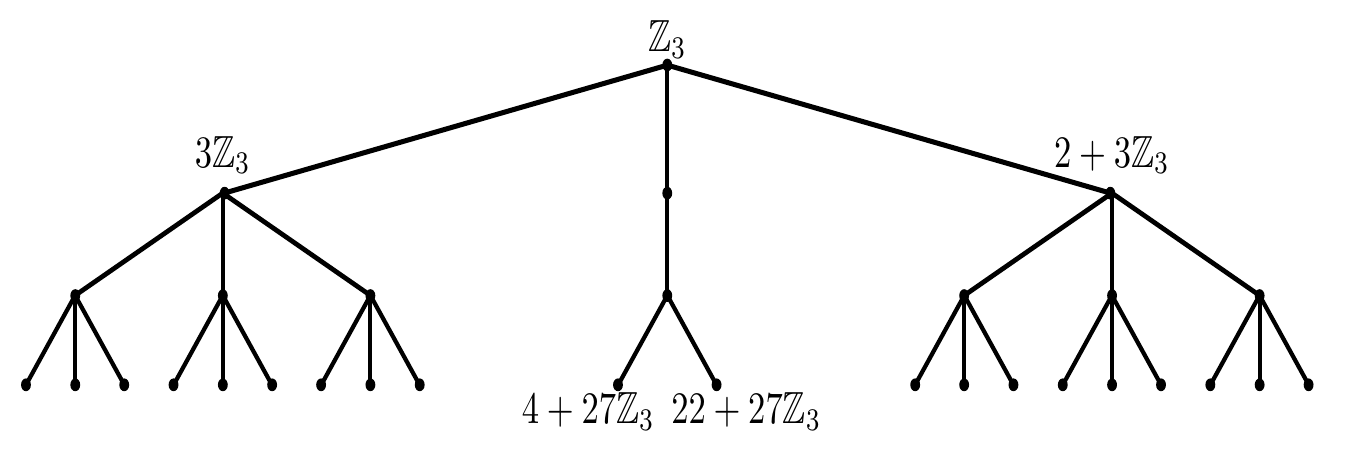}\\
  \caption{Consider the compact open set $O=3\Z_3\sqcup 3\Z_3 \sqcup 4+27\Z_3\sqcup  22+27\Z_3$  as a finite tree.}
\label{compactopen}
\end{figure}

We shall prove that  the Fuglede conjecture holds in $\Q_p$ among  compact open sets and that spectral sets are characterized by their $p$-homogeneity.

Notice that an open compact set $\Omega$ can be written as $\bigsqcup _{c\in C} (c+p^\gamma \Zp)$ for some finite set $C\subset \Qp$ and some integer $\gamma \in \Z$. As we shall see  in Section \ref{spectraltohomo}, for such a set $\Omega$ to be spectral with $\Lambda$ as spectrum if
$$
\forall \lambda,\lambda' \in \Lambda, \lambda\not= \lambda', \sum_{c\in C} \chi(-c(\lambda-\lambda'))=0 \hbox{ and  }\sharp(\Lambda \cap B(0,p^\lambda))=\sharp C.
$$
So we are led to study the trigomometric polynomial $\sum_{c\in C} \chi(ct)$.

\begin{theorem} \label{main}Let $\Omega $ be a compact open set in $\mathbb{Q}_p$. The following statements
 are equivalent: \\
  \indent {\rm (a)} \ $\Omega$ is a spectral set; \\
   \indent {\rm (b)} \ $\Omega$ is $p$-homogeneous; \\
 \indent {\rm (c)} \ $\Omega$ tiles $\Qp$ by translation.
\end{theorem}

For any  subset $\Omega\subset \Q_p$, the {\em set of admissible $p$-orders} of $\Omega$ is defined by $$I_{\Omega}:=\{i\in \Z:  \exists \  x,y \in \Omega \text { such that } v_p(x-y)=i  \}.$$
 Remark that  $p^{-I_{\Omega}}$ is the set of possible distances of 
different points in $\Omega$.
%Remark that the  Harr measure of  a ball in $\Q_p$ must be of form  $p^{-i}$ for some integer $i$. 

 Assume that $\Omega$ is a  $p$-homogeneous  compact open set. 
%We will find  the  geometrical structure of $\Omega$ is determined by the set $I_{\Omega}.$
By the definition of $I_\Omega$, an integer $i\in I_{\Omega}$ if and only if  the balls of radius $p^{-i}$ in the tree $\mathcal{T}_{\Omega}$  has $p$ descendants.
And  there is an integer  $\gamma$ such that $i\in I_{\Omega}$ if $i\geq \gamma$. This is the reason why could   a compact  open set  be described by  a finite tree. 

% An  isometry on $\Q_p$ do not change  the homogeneity of compact open sets  and their admissible $p$-order  sets.
% Furthermore,  for two   $p$-homogeneous  compact open sets $\Omega$ and $\Omega^{\prime}$,  
%there is an  isometric transformation $f:\Omega \to \Q_p$ such that $f(\Omega)=\Omega^{\prime}$ if  $I_{\Omega}=I_{\Omega^{\prime}}$ (similar  to   Lemma \ref{iso2}). 

%It is well know that $\Omega$ have 
On the other hand, it is of interest to investigate the structures of the spectra and the tiling complements of $\Omega$. We obtain  that the  spectra and  the tiling complements of $\Omega$ are uniquely determined by  the  set $I_{\Omega}$,  but subject to an isometric transformation  of $\Q_p$.

Set $\Z/ p\Z \cdot p^i :=\{ap^i: 0\le a \le p-1, a\in \mathbb{N} \}\subset \Qp$. 
Recall that  the addition of two subsets $A$ and $B$ in $\Q_p$ is defined by 
  $$A+B:=\{a+b: a\in A,b\in B \}.$$ 
  Let $\{A_i: i \in I\}$ be a family of subsets in $\Q_p$ such that all $A_i$ contain $0$.
 We  define 
$$
\sum_{i\in I} A_i:=\left\{\sum_{i\in J}a_i:J\subset I \text{  finite and }  a_i\in A_i  \right\}.
$$
\begin{theorem}\label{spetrumandtranslate}
Let  $\Omega$ be a  $p$-homogeneous compact open set in  $\Q_p$ with admissible $p$-order  set $I_{\Omega}$. 

 \indent {\rm (a)} Subject to an isometric bijection of $\Q_p$, 
 $$
 \Lambda=\sum_{\substack{i\in I_\Omega}} \Z/ p\Z \cdot p^{-i-1}
 $$
 %%$$\Lambda=\left\{ \sum _{i\geq n_0\in \Z}\lambda_i p^{i} \in \Q_p:  \lambda_{i}=0  \text{ if } -(i+1)\notin I_{\Omega}\} \right\}$$ 
 is the unique spectrum of $\Omega$. \\
 \indent {\rm (b)} Subject to an isometric bijection of $\Q_p$, 
 $$
 T=\sum_{i\notin I_\Omega} \Z/ p\Z \cdot p^{i}
 $$
 %%$$T=\left\{ \sum _{i\geq n_0\in \Z}\tau_i p^{i}\in \Q_p :  \tau_{i}=0  \text{ if } i\in I_{\Omega} \right\}$$ 
 is the unique  tiling complement of $\Omega$.
\end{theorem}

It is clear that if 
$\Omega$ is a spectral set with $\Lambda$ as spectrum, then so are its translates $\Omega + a$ ($a\in \mathbb{Q}_p$) with spectrum $\Lambda$
and its dilations $b \Omega $ ($b \in \mathbb{Q}_p^*$) with spectrum $b^{-1}\Lambda$. It is also true that the translation and the
dilation don't change the tiling property and the homogeneity.
Since $\Omega$ is compact open, by scaling and  translation, we may assume that
$\Omega\subset \Zp$ and $0\in\Omega.$  So it can be represented as a disjoint union of balls of same size
 $$\Omega=\bigsqcup_{c\in C}\left( c+ p^\gamma \Zp\right),$$ where $\gamma$ is a nonnegative integer and  $C\subset  \{0,1, \cdots,  p^{\gamma}-1\}$.

% Otherwise,we take any $a\in \Omega$ and  consider $p^{v} (\Omega -a)$ where$v = \min_{x\in \Omega} v_p(x-a)$, instead of $\Omega$.

 For each $0\leq n\leq \gamma$ , denote   by
 $$C_{\!\!\!\!\!\! \mod{p^n}} :=\{x\in \{0,1,\dots,p^n-1 \}: \exists~ y \in C, \text{ such that }  x\equiv y ~(\!\!\!\!\!\! \mod{p^n}) \}$$  the subset of  $\Z/p^n\Z$ determined by $C$ modulo $p^n$.

We also obtained  the following characterization of spectral sets in the finite group  $\mathbb{Z}/p^\gamma \mathbb{Z}$.

\begin{theorem}\label{1.3}
	Let $C\subset \mathbb{Z}/p^\gamma \mathbb{Z}$. The following statements are equivalent: \\
  \indent {\rm (a)} \ $C$ is a spectral set in $\mathbb{Z}/p^\gamma \mathbb{Z}$;\\
  \indent {\rm (b)} \	$C$ is a tile of $\mathbb{Z}/p^\gamma \mathbb{Z}$;\\
  \indent {\rm (c)} \	 For any $ n=1, 2, \cdots, \gamma -1$,
   $\sharp (C_{\!\!\!\! \mod{p^n}})=p^{k_n}$ for some integer $k_n\in \mathbb{N}$, where  $\sharp(C_{\!\!\!\! \mod{p^n}})$ is the cardinality of the finite set  $C_{\!\!\!\! \mod{p^n}}$.
\end{theorem}

As Terence Tao noted on page 3 of \cite{Tao2004}, the equivalence (a)$\Leftrightarrow$(b) is not new, but follows from Laba \cite{Laba2002}, which in turn builds on the work of Coven-Meyerowitz \cite{CM1999}.  However, we prove it by showing that both (a) and (b) are equivalent to (c) which implies $p$-homogeneity of $C$. The $p$-homogeneity is practically checkable. We can use it to describe finite spectral sets (more precisely, the probability  spectral measures  uniformly distributed on  finite sets) in $\Qp$.  A probability Borel measure $\mu$ on $\Q_p$ is called a {\em spectral measure}
if there exists a set $\Lambda \subset \widehat{\Q}_p$ such that
$\{\chi_\lambda\}_{\lambda \in \Lambda}$ forms an  orthonormal basis (i.e. a Hilbert basis) of the $L^2(\mu)$.
Let $F$ be a finite subset  of $\Q_p$. Consider the uniform probability measure on $F$
 defined by
$$\delta_{F} := \frac{1}{\sharp{F}}\sum_{c\in F} \delta_{c},$$
where $\delta_{c} $ is the dirac measure concentrated at the point $c$.
Let
$$  \gamma_F =  \max_{\substack {c,c^{\prime}\in F\\c\neq c^{\prime}} }v_p(c- c^{\prime}),$$
where $v_p(x)$ denotes the $p$-valuation of $x\in \mathbb{Q}_p$. Then $p^{-\gamma_F}$ is  the minimal distance between different points in $F$. %The number $\gamma_F$ is an integer.

\begin{theorem} \label{discrete}
The measure $\delta_{F} $ is a spectral measure if and only if for each  integer $\gamma>\gamma_F$,
the compact open set  $\Omega_{\gamma}:=\bigsqcup_{c\in F} B(c,p^{-\gamma})$ is a spectral set.
\end{theorem}
The above theorem provides a criterion of finite spectral set, combining with Theorem \ref{main}. 
%Denote by
%$$Dist(F):=\{|c-c^{\prime}|_p: c, c^{\prime}\in F \hbox{ and } c\neq c^{\prime}\}$$
%the set of possible distances of different points in $F$. 

\begin{corollary}The measure $\delta_{F} $ is a spectral measure if and only if  
$$\sharp F=p^{\sharp{I_F}}.$$
\end{corollary}
Moreover, we are interested in finding more  spectral measures. And we provide a class of Cantor spectral  measures.
\begin{theorem}\label{1.6}
There exists a class of singular spectral measures in $\Qp$.
\end{theorem}

The article is organized  as follows. In Section 2, we introduce some basic definitions and preliminaries on the field $\Q_p$ of $p$-adic numbers, Fourier analysis on $\Q_p$ and $\mathbb{Z}$-module generated by the $p^n$-th roots of unity.  In Section 3, we prove Theorem \ref{main}. In Section 4, we characterize spectral sets and tiles in the finite group $\Z/p^\gamma \Z.$ Theorem \ref{1.3} is proved there.  Section 5 is devoted to the characterization of spectra and tiling complements. Theorem \ref{spetrumandtranslate} is proved at the end of this section. In Section 6,  we characterize finite spectral sets in $\mathbb{Q}_p$(Theorem \ref{discrete}). In  Section 7, we shall construct  a class of singular spectral measures (Theorem \ref{1.6}) and present two concrete examples .

\section{Preliminaries}
In this section, we present some preliminaries.
Some of them have their own interests, like characterization of spectral measures using Fourier
transform, $\mathbb{Z}$-module generated by the $p^n$-th roots of unity, uniform distribution
of spectrum etc. We start with  recalling  $p$-adic numbers and related notation, and the
computation of the Fourier transform of the indicator function of a compact open set.

\subsection{The field of $p$-adic numbers}
Consider the field $\mathbb{Q}$ of rational numbers and a prime $p\geq 2$.
Any nonzero rational number $r\in \mathbb{Q}$ can be written as
$r =p^v \frac{a}{b}$ where $v, a, b\in \mathbb{Z}$ and $(p, a)=1$ and $(p, b)=1$
(here $(x, y)$ denotes the greatest common divisor of two integers $x$ and $y$).  We define $v_p(r)=v$ and
$|r|_p = p^{-v_p(r)}$ for $r\not=0$ and $|0|_p=0$.
Then $|\cdot|_p$ is a non-Archimedean absolute value on $\Q$. That means\\
\indent (i)  \ \ $|r|_p\ge 0$ with equality only for $r=0$; \\
\indent (ii) \ $|r s|_p=|r|_p |s|_p$;\\
\indent (iii) $|r+s|_p\le \max\{ |r|_p, |s|_p\}$.\\
The field of $p$-adic numbers $\mathbb{Q}_p$ is the completion of $\mathbb{Q}$ under the absolute value 
$|\cdot|_p$. Actually
a typical element of $\mathbb{Q}_p$ is of the form of a convergent series
$$
     \sum_{n= N}^\infty a_n p^{n} \qquad (N\in \mathbb{Z}, a_n \in \{0,1,\cdots, p-1\}, a_{N}\neq 0).
$$
Consider $\mathbb{Q}_p$ as an additive group. Then a non-trivial group character is the following function
$$
    \chi(x) = e^{2\pi i \{x\}},
$$
where $\{x\}= \sum_{n=N}^{-1} a_n p^n$ is the fractional part of $x=\sum_{n=N}^{\infty} a_n p^n$. From this character we can obtain all characters $\chi_y$ of $\mathbb{Q}_p$, which are defined by
$\chi_y(x) =\chi(yx)$.
Remark that  each $\chi_y(\cdot)$  is uniformly locally  constant,  i.e.
$$\chi_y(x)=\chi_y(x^{\prime}), \hbox{ if } |x-x^{\prime}|_p\leq  \frac{1}{|y|_p} . $$ The interested readers are referred  to \cite{Taibleson1975,Vladimirov-Volovich-Zelenov1994}  for  further information about characters of $\Q_p$.

\medskip

\noindent Notation:
\\ \indent
$\Zp^\times := \Zp\setminus p\Zp=\{x\in \Qp: |x|_p=1\}$.
It is the group of units of $\Zp$.

%$\mathfrak{U}_n:= 1 + \mathfrak{p}^n \mathfrak{D}$ ($n\ge 1$).
%These are subgroups of $\mathfrak{D}^\times$.

$B(0, p^{n}): = p^{-n} \Zp$.  It is the (closed) ball centered at $0$ of radius $p^n$.

$B(x, p^{n}): = x + B(0, p^{n})$. %We also use it to denote balls in $\Qp$.

%$S(x, q^{n}): = B(x, q^{n}) \setminus B(x, q^{n-1}) $, the sphere  of radius $q^n$.

%$\mathcal{A}_n:$ the set of finite union of balls of radius $p^n$ ($n\in \mathbb{Z})$.

$1_A:$ the indicator function of a set $A$.

%\medskip
%\noindent Facts.
%\\ \indent
%$\mathfrak{m}(\mathfrak{P})= q^{-1}$, $|\mathfrak{p}|=q^{-1}$,
%$\mathfrak{m}(B(x, q^n)= q^n$.

%All $\mathfrak{p}^n \mathfrak{D}$ are additive groups.

%All $\mathfrak{U}_n$ are multiplicative groups.

%$d \mathfrak{m}(ax) =|a|d\mathfrak{m}(x)$ for all $a\in K^*$. It is the image of $\mathfrak{m}$
%under $x\mapsto ax$.

\subsection{Fourier transformation}

Let $\mu$ be a finite Borel measure on $\Qp$. The {\em Fourier transform} of
$\mu$ is classically defined to be
$$
    \widehat{\mu} (y)
    = \int_{\Qp} \overline{\chi}_y(x) d\mu(x) \qquad (y \in \widehat{\Q}_p \simeq \Qp).
$$
The Fourier transform $\widehat{f}$ of $f \in L^1(\Qp)$ is that of $\mu_f$ where $\mu_f$ is the measure
defined by $d\mu_f = f d\mathfrak{m}$.
 % Fourier analysis on local field, we can refer to \cite{Taibleson1975}.

The following lemma shows that the Fourier transform of the indicator function of a ball centered at $0$
is a function of the same type and the Fourier transform of the indicator function of  a compact open set
is also supported by a ball, and on the ball it is the restriction of a trigonometric polynomial.

 \begin{lemma}\label{FourierIntegral} Let  $\gamma\in \mathbb{Z}$ be an integer. \\
 \indent {\rm (a)}\  We have   $\widehat{1_{B(0, p^\gamma)}}(\xi)= p^\gamma 1_{B(0, p^{-\gamma})} (\xi)$ for all $\xi \in \Qp$.\\
 \indent {\rm (b)}\ If $\Omega = \bigsqcup_j B(c_j, p^\gamma)$ is a finite union of disjoint  balls of the same size, then  
 \begin{align}\label{fourier}
 \widehat{1_\Omega}(\xi) = p^\gamma 1_{B(0, p^{-\gamma})} (\xi) \sum_j \chi(- c_j\xi ).
 \end{align}
 
 \end{lemma}

 \begin{proof} (a) By the scaling property of the Haar measure, we have only  to prove the result
 in the case $\gamma=0$. Recall that
 $$
    \widehat{1_{B(0, 1)}}(\xi)
    = \int_{B(0, 1)} \chi(-\xi  x) dx.
 $$
 When $|\xi|_p\le 1$, the integrand is equal to $1$, so $\widehat{1_{B(0, 1)}}(\xi)=1$.
 When $|\xi|_p>1$, making a change of variable $x = y-z$ with $z\in B(0, 1)$ chosen
 such that $ \chi(\xi\cdot z)\not=1$,
 we get $$\widehat{1_{B(0, 1)}}(\xi) = \chi(\xi z) \widehat{1_{B(0, 1)}}(\xi).$$
 It follows that $\widehat{1_{B(0, 1)}}(\xi)=0$ for $|\xi|_p>1$.

 (b) is  a direct consequence of (a) and of the fact
 $$
 \widehat{1_{B(c, p^r)}} (\xi) = \chi(-c \xi)   \widehat{1_{B(0, p^r)}} (\xi).
 $$
 \end{proof}

 % \begin{lemma}\label{fourierlem}
 % Let $\Omega = \bigsqcup_j B(c_j, p^\gamma)\in \mathcal{A}_\gamma$ be a finite union of ball of %the same size,
 %where $\gamma\in \mathbb{Z}$. We have
 %\begin{align}\label{fourier}
 %\widehat{1_\Omega}(\xi) = p^\gamma 1_{B(0, p^{-\gamma})} (\xi) \sum_j \chi(- \xi c_j).
 %\end{align}
 %In particular, $\widehat{1_\Omega}(\xi)$ is supported by the ball $B(0, p^{-\gamma})$.
 %\end{lemma}
 %\begin{proof}
 %\end{proof}

\subsection{ A criterion of spectral measure} %Let $\mu$ be a finite Borel measure on $\Qp$.
Let $\mu$ be a probability Borel measure on $\Qp$. We say that $\mu$ is a {\em spectral measure}
if there exists a set $\Lambda \subset \widehat{\Q}_p$ such that
$\{\chi_\lambda\}_{\lambda \in \Lambda}$ is an orthonormal basis (i.e. a Hilbert basis) of $L^2(\mu)$. Then
$\Lambda$ is called a {\em spectrum} of $\mu$ and we call $(\mu, \Lambda)$ a {\em spectral pair}.
Assume that $\Omega$ is a set in $\Qp$
of positive and finite Haar measure. That $\Omega$ is a {\em spectral set} means the restricted measure $\frac{1}{\mathfrak{m}(\Omega)}\mathfrak{m}|_\Omega$
is a spectral measure. In this case, instead of saying $(\frac{1}{\mathfrak{m}(\Omega)}\mathfrak{m}|_\Omega, \Lambda)$ is a spectral pair,  we say that
$(\Omega, \Lambda)$ is a {\em spectral pair}.

Here is a criterion for a probability measure $\mu$ to be a spectral measure.
The result in the case $\mathbb{R}^d$
is due to Jorgensen and Pedersen \cite{Jorgensen-Pedersen1998}.
%Actually, for the sufficiency of (\ref{spectral criterion}), we only
%need the equality for $\xi\in \Lambda$.
%The following theorem says that a Bessel sequence perturbed by a bounded sequence
%remains a Bessel sequence. The corresponding result in $\mathbb{R}^d$  was proved by Dutkay, Han, %Sun and Weber in \cite{Dutkay-Han-Sun-Weber2011}.
%The proof in $\mathbb{R}^d$ seems not adaptable to the case of local fields. Our proof will based on the fact that characters in local fields are constant in a neighborhood of the origin.
It holds on any local field (see \cite{Fan}). The proof is the same as in the Euclidean space.
We reproduce the proof here for completeness.

\begin{lemma}\label{Thm-SpectralMeasure} A Borel probability measure on $\Qp$ is a spectral measure with $\Lambda \subset
\widehat{\Q}_p$ as its spectrum iff
\begin{equation}\label{spectral criterion}
\forall \xi \in \widehat{\Q}_p, \quad \sum_{\lambda \in \Lambda}
 |\widehat{\mu} (\lambda -\xi)|^2 = 1.
 \end{equation}
 In particular, a Borel set $\Omega$ such that $0<|\Omega|<\infty$ is
   a spectral set with $\Lambda$ as spectrum iff
   \begin{align}\label{equ}
   \forall \xi \in \widehat{\Q}_p, \quad \sum_{\lambda \in \Lambda}
 |\widehat{1_\Omega} (\lambda -\xi)|^2 =|\Omega|^2.
   \end{align}
\end{lemma}

\begin{proof}
  Recall that $\langle f, g\rangle_\mu$ denotes the inner product in $L^2(\mu)$:
  $$
     \langle f, g\rangle_\mu = \int f \overline{g} d \mu, \quad \forall f, g \in L^2(\mu).
  $$
 Remark that
  $$
      \langle \chi_\xi, \chi_\lambda\rangle_\mu = \int \chi_\xi \overline{\chi}_\lambda d\mu = \widehat{\mu}(\lambda -\xi).
  $$
  It follows that $\chi_{\lambda}$ and $\chi_{\xi}$
  are orthogonal in $L^2(\mu)$ iff $\widehat{\mu}(\lambda -\xi)=0$.

  Assume that $(\mu, \Lambda)$ is a spectral pair. Then  (\ref{spectral criterion}) holds because of the Parseval
  equality and of the fact that $\{\widehat{\mu}(\lambda -\xi)\}_{\lambda\in \Lambda}$
  are Fourier coefficients of $\chi_\xi$ under the Hilbert basis $\{\chi_\lambda\}_{\lambda\in \Lambda}$.

  Now assume (\ref{spectral criterion}) holds. Fix any $\lambda'\in \Lambda$ and
   take $\xi=\lambda'$ in (\ref{spectral criterion}). We get
   $$
       1 + \sum_{\lambda\in \Lambda, \lambda\not=\lambda'}|\widehat{\mu}(\lambda -\lambda')|^2=1,
       $$
   which implies  $\widehat{\mu}(\lambda -\lambda')=0$ for all $\lambda \in \Lambda\setminus \{\lambda'\}$. Thus we have proved the orthogonality of
   $\{\chi_\lambda\}_{\lambda \in \Lambda}$.
   It remains to prove that
  $\{\chi_\lambda\}_{\lambda\in \Lambda}$ is complete. By the Hahn-Banach Theorem, what we have to prove is the implication
  $$
      f\in L^2(\mu), \forall \lambda \in  \Lambda, \langle f, \chi_\lambda\rangle_\mu=0
      \Rightarrow f=0.
  $$
  The condition (\ref{spectral criterion}) implies that
  $$
     \forall \xi\in \widehat{\Q}_p, \quad \chi_\xi = \sum_{\lambda\in \Lambda}
      \langle \chi_\xi, \chi_\lambda\rangle_\mu \chi_\lambda.
  $$
  This implies that $\chi_\xi$ is in the closure of the space spanned by
  $\{\chi_\lambda\}_{\lambda\in \Lambda}$. As $f$ is orthogonal to $\chi_\lambda$
  for all $\lambda \in \Lambda$. So, $f$ is orthogonal to
  $\chi_\xi$. Thus we have proved that
  $$
     \forall \xi \in \widehat{\Q}_p, \quad \int \overline{\chi}_\xi f d \mu = \langle f, \chi_\xi\rangle_\mu = 0.
  $$
  That is, the Fourier coefficients of the measure $f d\mu$ are all zero. Finally
  $f=0$ $\mu$-almost everywhere.
\end{proof}
%\subsection{Spectral sets}

\subsection{$\mathbb{Z}$-module generated by $p^n$-th roots of unity}
The Fourier condition of  a spectral set is tightly related to the fact that certain sums of roots of unity must be zero. Let $m\ge 2$ be an integer and let $\omega_m = e^{2\pi i/m}$, which is a primitive $m$-th root of unity. Denote  by $\mathcal{M}_m$  the set of integral points
$(a_0, a_1, \cdots, a_{m-1}) \in \mathbb{Z}^m$ such that
$$
   \sum_{j=0} ^{m-1} a_j \omega_m^j =0,
$$
which form a $\mathbb{Z}$-module. The fact that the degree over $\mathbb{Q}$ of the extension field $\mathbb{Q}(\omega_m)$
is equal to $\phi(m)$, where $\phi$ is Euler's totient
function, implies  that the rank of $\mathcal{M}_m$ is equal to
$m -\phi(m)$. Schoenberg (\cite{Schoenberg1964}, Theorem 1) found a set of generators. See also de Bruijn
\cite{Br1953} and R\'edei \cite{redei1950,redei1954}(actually,  R\'edei \cite{redei1950} was the first to formulate this, his incomplete proof was completed in \cite{Br1953} and in \cite{redei1954}, and later independently in \cite{Schoenberg1964}). Lagarias and Wang (\cite{Lagarias-Wang1996}, Lemma 4.1) observed that this set of generators is actually a base
when $m$ is a prime power.  Let $p$ be a prime and  $n$ be a positive integer.

\begin{lemma}[\cite{Lagarias-Wang1996,Schoenberg1964}]\label{root}
Let $(a_0,a_1,\cdots, a_{p^n-1})\in  \mathcal{M}_{p^n}$.
%Suppose  $$\sum_{i=0}^{p^n-1 }a_i\omega_{p^n}^i=0.$$ 
Then  for any integer $0\leq i\leq p^{n-1}-1$,  we have $a_i=a_{i+j p^{n-1}}$ for all $j=0,1,\cdots, p-1$.
 \end{lemma}
 %For the proof of Lemma \ref{root}  we refer the reader to Lemma 4.1 of \cite{Lagarias-Wang1996}.

 We shall use Lemma \ref{root} in the following two particular forms. The first one is an immediate
 consequence.

 \begin{lemma}\label{permu}
 Let $(b_0,b_1,\cdots, b_{p-1})\in \Z^{p}$.
 If $\sum_{j=0}^{p-1} e^{2\pi i b_j/p^n}=0$, then subject to a permutation of $(b_0,\cdots, b_{p-1})$, there exists  $0\leq r \leq p^{n-1}-1$ such  that
$$b_j  \equiv r+ jp^{n-1} (\!\!\!\!\mod p^n) $$ for all  $ j =0,1, \cdots,p-1$.
 \end{lemma}

 \begin{lemma}\label{C}
 Let $C$ be a finite subset of $\Z$. If $\sum_{c\in C} e^{2\pi i c/p^n}=0$, then $p \mid \sharp C$ and $C$ can be decomposed into  $\sharp C/p$ disjoint subsets
 $C_1,  C_2,  \cdots, $ $ C_{\sharp C/p}$, such that each $C_j$ consists of $p$ points and
 $$\sum_{c\in C_j}e^{2\pi i c/p^n}=0.$$
\end{lemma}
\begin{proof} Observe that $e^{2\pi i c/p^n}=e^{2\pi i c^{\prime}/p^n}$ if and only if $c\equiv c^{\prime}  (\!\!\!\!\mod p^n)$.
Fix a point $c_0\in C$. By Lemma \ref{root}, there are  other $p-1$ points $c_1,c_2,\cdots, c_{p-1}\in C$
such that $c_j \equiv c_0+jp^{n-1} (\!\!\!\!\mod p^n)$ for all $1\leq j \leq p-1$.   Thus we have
$$\sum_{0\leq i\leq p-1}e^{2\pi i c_j/p^n}=0.$$ Set $C_1=\{c_0,c_1,\cdots,c_{p-1}\}$. So, the hypothesis is reduced to  $$\sum_{c\in C\setminus C_1} e^{2\pi i c/p^n}=0.$$
 We can complete the proof by induction.
\end{proof}

The following lemma states that  the property
$\sum_{j=0}^{m-1}\chi(\xi_j)=0$ of the set of points
$(\xi_0, \xi_1, \cdots, \xi_{m-1})\in \Qp^m$  is invariant under `rotation'.

 \begin{lemma} \label{multi}
 Let $\xi_0, \xi_1,\cdots,  \xi_{m-1}$ be $m$ points in $\Qp$.
 If $\sum_{j=0}^{m-1}\chi(\xi_j)=0,$ then $p\mid m$ and  $$\sum_{j=0}^{m-1}\chi(x  \xi_j)=0$$
 for all $x\in \Zp^{\times}$.
\end{lemma}
\begin{proof}By  Lemma \ref{C}, we get $p\mid m$ and moreover, $\{\xi_0, \xi_1,\cdots, \xi_{m-1}\}$  consists of  $m/p$ subsets  $C_1, C_2, \cdots,  C_{m/p}$ such that  each $C_j$ ,$1\leq j\leq m/p $, contains $p$ elements and
$\sum_{\xi\in C_j}\chi(\xi)=0$.

Without  loss of generality, we assume that $m=p$. By Lemma \ref{permu}, subject to a permutation of $(\xi_0,\cdots,\xi_{p-1})$,
there exists $r \in \Q_p$ such  that
 $$\xi_j \equiv r+ j/p \ (\!\!\!\!\!\mod \Z_p) $$ for all  $ j=0,1, \cdots,p-1$.
 Now, for any given $x\in \Z_p^{\times}$,
 we have $$x \xi_j \equiv xr+\frac{x_0j}{p} \ (\!\!\!\!\!\mod \Z_p)$$
where $x_0\in \{1, \cdots, p-1\}$ is the first digit of the $p$-adic expansion of $x$.
Observe that the multiplication by $x_0$ induces  a permutation on $\{0, 1, \cdots,  p-1\}$.
So we have
$$\sum_{j=0}^{p-1}\chi(x \xi)=\sum_{k=0}^{p-1}e^{2\pi i\{xr+\frac{k}{p}\}}=0.$$
\end{proof}

\subsection{Uniform distribution of spectrum}

 The following lemma establishes the fact that  given a compact open spectral set in $\mathbb{Z}_p$ consisting of
  small balls of radius $p^{-\gamma}$ ($\gamma >0$), any spectrum of the set is uniformly distributed in the sense that
    any ball of radius $p^\gamma$ contains exactly as many points as the number of small balls
    of radius $p^{-\gamma}$ in the spectral set.
  This fact will contribute to proving  ``spectral property implies homogeneity" of Theorem \ref{main}.
  
Let $\Omega$ be a compact open subset of $\Zp$. Assume that $\Omega$ is of the form $\Omega=\bigsqcup_{c\in C}c+ p^\gamma \Zp$, where $\gamma$ is a positive integer and  $C\subset  \{0,1, \cdots,  p^{\gamma}-1\}$.
\begin{lemma}\label{number}
Suppose that  $(\Omega,\Lambda)$ is a spectral pair.
Then every closed ball of  radius $p^\gamma$ contains $\sharp C$ spectrum points in $\Lambda$. That is,   $$\sharp(B(a,{p^\gamma})\cap \Lambda)= \sharp C$$ for every  $a \in \Qp$.
\end{lemma}
 \begin{proof} By Lemma \ref{FourierIntegral}, we have
 $$\widehat{1_{\Omega}}(\lambda-\xi)=p^{-\gamma} 1_{B(0,p^\gamma)}(\lambda-\xi)\sum_{c\in C}\chi(-(\lambda-\xi)c).$$
 Then a simple computation leads to
 $$\sum_{\lambda\in \Lambda} |\widehat{1_{\Omega}}(\lambda-\xi)|^2=\sum_{\lambda\in \Lambda}p^{-2\gamma}1_{B(0, p^\gamma)}(\lambda-\xi)
    \Big( \sharp C+\sum_{\substack{c,c^{\prime} \in C \\c \neq c^{\prime}}}\chi((c-c^\prime)(\lambda-\xi))\Big).$$
Consider the equality (\ref{equ}) in Lemma \ref{Thm-SpectralMeasure}.
By integrating  both sides of this equality on the ball  $B(a, p^\gamma)$, we get
\begin{align}\label{1}
|\Omega|^2 p^\gamma
               &=p^{-2\gamma} \sum_{\substack {\lambda\in \Lambda\\ |\lambda-a|\leq p^\gamma}}\Big(\sharp C p^\gamma+\sum_{\substack{c,c^{\prime} \in C \\c \neq c^{\prime}}} \int_{B(a,p^\gamma)} \chi((c-c^\prime)(\lambda-\xi)) d\xi\Big).
\end{align}
Here we have used the fact that two balls of same size are either identical or disjoint.
 Observe that
 \begin{align}\label{2}
 \int_{B(a,p^\gamma)}\chi((c-c^\prime)(\lambda-\xi))d\xi =\chi((c-c^\prime)\lambda )\cdot \widehat{1_{B(a,p^\gamma)}}(c-c^{\prime})
 \end{align}
 and
 $\widehat{1_{B(a,p^\gamma)}}(c-c^{\prime})=\chi(-(c-c^\prime)a)\cdot p^{\gamma}1_{B(0,p^{-\gamma)}}(c-c^{\prime})$.
 However, by the assumption, $|c-c^{\prime}|_p >p^{-\gamma}$ for $c,c^{\prime}\in C$ with $c\neq c^{\prime}$. Hence we  have
 \begin{align}\label{3}
 \widehat{1_{B(a,p^\gamma)}}(c-c^{\prime})&=0.
 \end{align}
 Applying the equalities (\ref{2})  and  (\ref{3}) to the equality (\ref{1}), we obtain
 \begin{align}
   |\Omega|^2\cdot p^\gamma   &=\sharp C\cdot p^{-\gamma}\cdot \sharp (\Lambda \cap B(a,p^\gamma))
\end{align}
Since $|\Omega|=\sharp C\cdot p^{-\gamma}$, we finally get
 $\sharp (\Lambda \cap B(a,p^\gamma))=\sharp C$.

 \end{proof}

 The restriction that $\Omega$ is contained in $\mathbb{Z}_p$ is not necessary, because scaling and translation 
preserves the spectral property.

\subsection{ Finite $p$-homogeneous trees}
\label{Finitehomogeneoustree}
Let $\gamma $ be a positive integer. 
%Through the development
%$$
%t=t_0+t_1p+\cdots+t_{\gamma-1 }p^{\gamma-1},
%$$
To any $t_0t_1\cdots t_{\gamma -1}\in \{0, 1, 2,  \cdots, p-1\}^\gamma $,  we associate
an integer
    $$
        c =c(t_0t_1\cdots t_{\gamma -1})=\sum_{i=0}^{\gamma -1}
               t_i p^i \in \{0, 1, 2, \cdots, p^\gamma -1\}.
    $$
So  $ \mathbb{Z}/p^\gamma\mathbb{Z}\simeq\{0, 1, \cdots, p^\gamma-1\}$  is identified with $\{0, 1, 2,  \cdots, p-1\}^\gamma$ which is considered as a finite tree, denoted by $\mathcal{T}^{(\gamma)}$, see {\sc Figure} \ref{Tn1} for an example. 
 The set of vertices $\mathcal{T}^{(\gamma)}$ consists of  the disjoint union  of the sets $\mathbb{Z}/p^n\mathbb{Z}, 0\leq n\leq \gamma$. Each  vertex, except the root of the tree,  is identified with a sequence $t_0t_1\cdots t_{n-1}$ with $1\leq n\leq\gamma$ and $t_i\in \{0,1, \cdots, p-1\}$.
 The set of edges consists of pairs $(x,y)\in \mathbb{Z}/p^n\mathbb{Z}\times \mathbb{Z}/p^{n+1}\mathbb{Z}$, such that $x\equiv y ~(\!\!\!\!\mod p^n)$, where $0\leq n\leq \gamma-1$.  For example, each
point $c$ of $\mathbb{Z}/p^\gamma\mathbb{Z}$ is identified with $\sum_{i=0}^{\gamma -1} t_i p^i \in \{0, 1, \cdots, p^\gamma-1\}$, which is called {\em a boundary point} of the tree. 
 \begin{figure}[H]
  \centering
  \includegraphics[width=\textwidth]{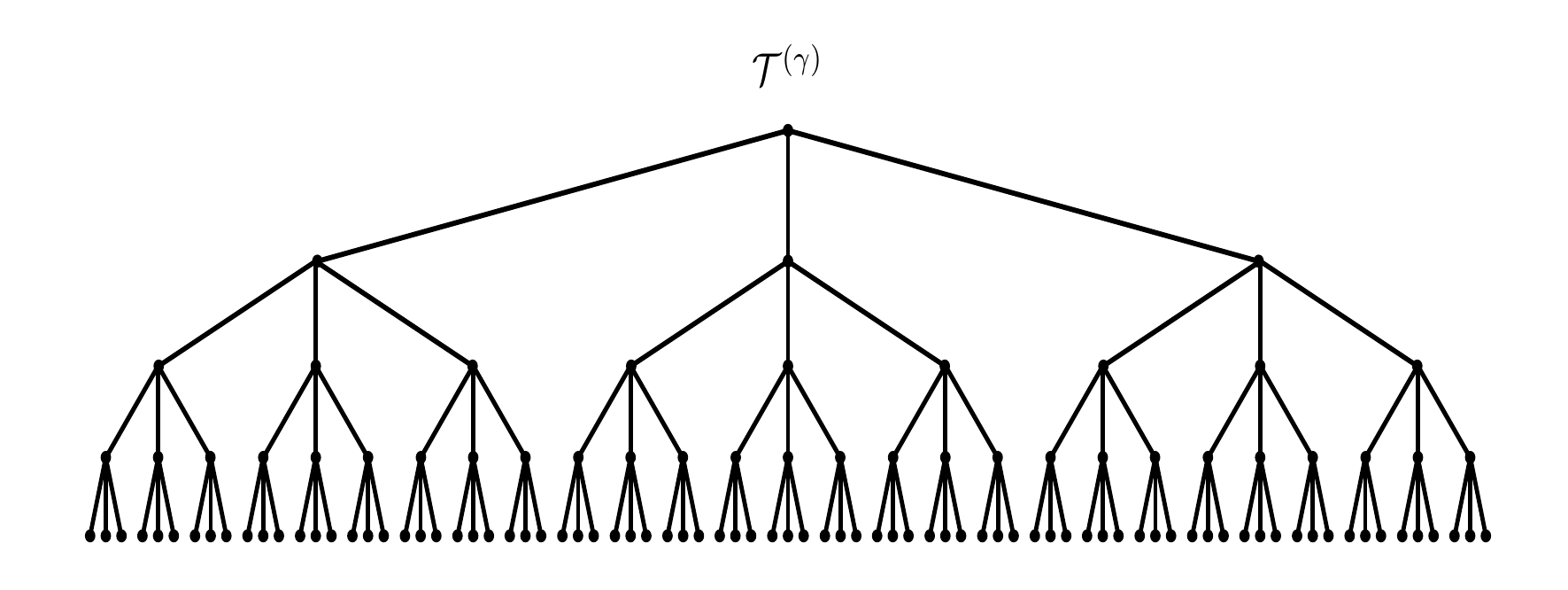}\\
  \caption{ The set  $\mathbb{Z}/3^4\mathbb{Z}=\{0,1,2,\cdots, 80\}$ is considered as a tree $\mathcal{T}^{(4)}$.}
\label{Tn1}
\end{figure}

Each subset $C\subset \mathbb{Z}/p^\gamma\mathbb{Z}$ will determine a subtree of $\mathcal{T}^{(\gamma)}$, denoted by $\mathcal{T}_{C}$, which consists  of 
the paths from the root  to the boundary points in $C$. 
  The set of vertices $\mathcal{T}_{C}$ consists of  the disjoint union  of the sets $C_{\!\!\!\! \mod{p^n}}, 0\leq n\leq \gamma$.  The set of edges consists of pairs $(x,y)\in C_{\!\!\!\! \mod{p^n}}\times C_{\!\!\!\! \mod{p^{n+1}}}$, such that $x\equiv y ~(\!\!\!\!\mod p^n)$, where $0\leq n\leq \gamma-1$. 

Now we are going to construct a class of subtrees of $\mathcal{T}^{(\gamma)}$.
Let $I$ be a subset of $\{0, 1, 2, \cdots, \gamma-1\}$ and let $J$ be the complement
of $I$ in  $\{0, 1, 2, \cdots, \gamma-1\}$. Thus $I$ and $J$ form a
partition of $\{0, 1, 2, \cdots, \gamma-1\}$.
It is allowed that $I$ or $J$ is empty.
% We define a subtree $\mathcal{T}_{I, J}$ of $\mathcal{T}^{(\gamma)}$ as follows. 
We say a subtree of $\mathcal{T}^{(\gamma)}$ is of  {\em $\mathcal{T}_{I, J}$-form} if its boundary points
$t_0t_1\cdots t_{\gamma-1}$ of  $\mathcal{T}_{I, J}$ are those of $\mathcal{T}^{(\gamma)}$ satisfying the following conditions:\\
\indent (i) \ \  if $i \in I$, $t_i$ can take any value of $\{0, 1, \cdots, p-1\}$;\\
\indent (ii) \ if $i\in J$, for any $t_0t_1\cdots t_{i-1}$, we fix one value of $\{0, 1, \cdots, p-1\}$ which is the only value taken by $t_i$. In other words, $t_i$ takes only one value from $\{0, 1, \cdots, p-1\}$ which depends on $t_0t_1\cdots t_{i-1}$.

Remark that such a subtree depends not only on $I$ and $J$ but also on the values taken
by $t_i$'s with $i\in J$. 
%For simplicity, this dependance is not reflected in the
%notation $\mathcal{T}_{I, J}$.
 A $\mathcal{T}_{I, J}$-form tree  is  called a  finite $p$-{\em homogeneous tree}.
A special subtree $\mathcal{T}_{I, J}$ is shown in {\sc Figure} \ref{TIJ}.
  \begin{figure}
  \centering
  \includegraphics[width=1\textwidth]{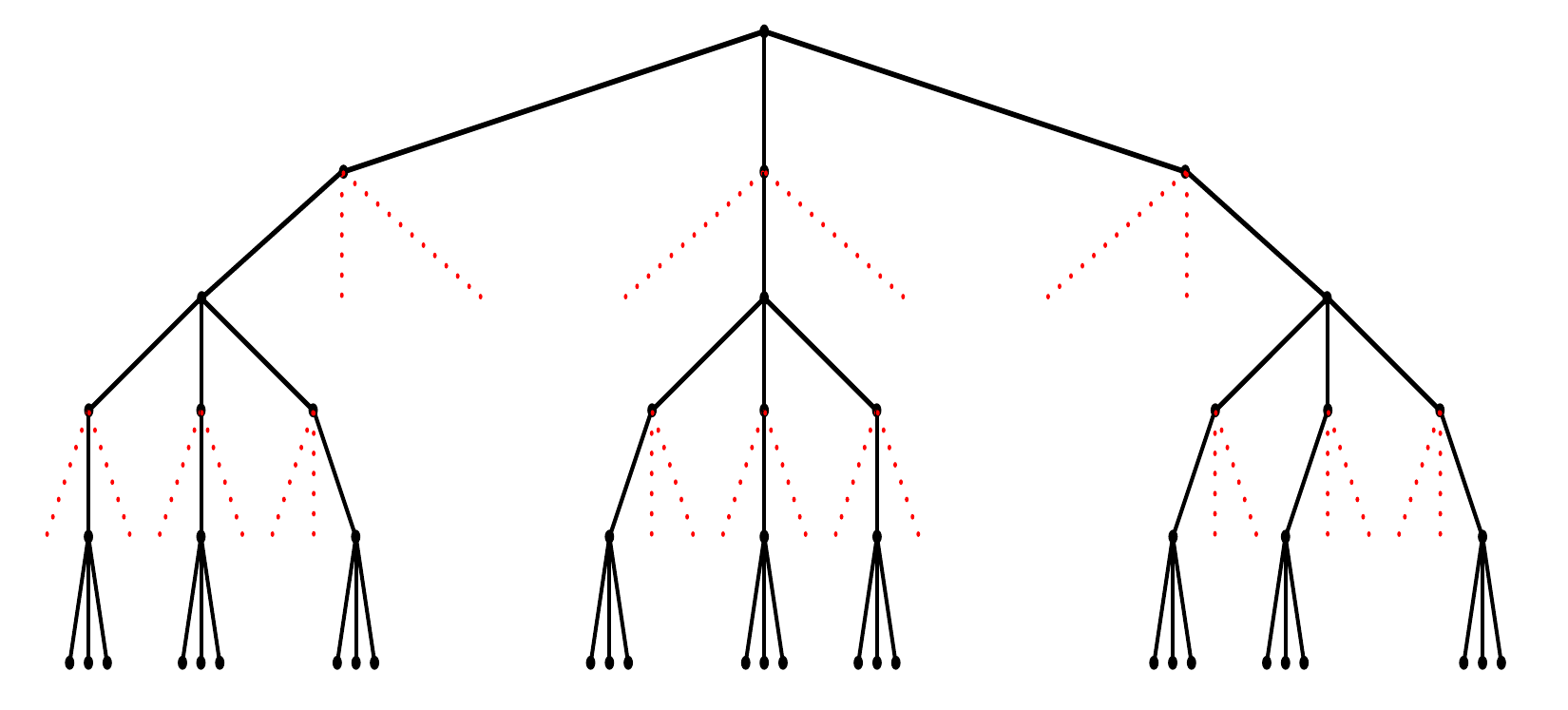}\\
  \caption{For $p=3$, a $\mathcal{T}_{I,J}$-form tree  with $\gamma=5, I=\{0, 2, 4\} , J=\{1,3\}$.}
\label{TIJ}
\end{figure}
  
A set $C\subset \Z/p^{\gamma}\Z$ is said to be {\em $p$-homogeneous} if the corresponding tree $\mathcal{T}_{C}$ is  $p$-homogeneous.
If   $C\subset \{0,1,2,\cdots, p^{\gamma}-1\}$ is considered as a subset of $\Z_p$, the tree $\mathcal{T}_{C}$ could be identified with the finite tree determined  by the compact open set $\Omega=\bigsqcup_{c\in C}c+ p^\gamma \Zp$. By definition, we immediately have the following lemma.
\begin{lemma}\label{b}
The above compact set $\Omega$ is  $p$-homogeneous in $\Qp$ if and only if the finite set $C\subset \mathbb{Z}/p^\gamma \mathbb{Z}$ is 
$p$-homogeneous.
\end{lemma}

An  algebraic criterion for  the $p$-homogenity  of a set $C\subset \Z/p^{\gamma}\Z$ is presented in the following theorem.

\begin{theorem}\label{keytheorem}
	Let $\gamma$ be a positive integer and let $C\subset \Z/p^\gamma \Z$. Suppose  (i) $\sharp C\le p^ n$ for some integer $1\le n\le \gamma$; (ii) there exist $n$ integers  $1\le i_1 <i_2< \cdots <i_n\le \gamma$ such that     
	\begin{equation}\label{keyeq}
	\sum_{c \in C} e^{2 \pi i c p^{-i_k}} =0  \hbox{ for all } 1\le k\le n.
	\end{equation}
	Then $\sharp C= p^ n$ and $C$  is $p$-homogeneous.  Moreover, $\mathcal{T}_{C}$ is a   $\mathcal{T}_{I, J}$-form tree with $I =\{i_1-1,  i_2-1, \cdots, i_n-1\}$ and $J= \{0,1, \cdots, \gamma-1\}\setminus I$. 
\end{theorem}
\begin{proof}
For simplicity,  let $m=\sharp C$.
	% Firstly, look at \ref{keyeq} with $k=n$.
	  By Lemma \ref{C} and the equality (\ref{keyeq}) with $k=n$, $p\mid m$ and  $C$ can be decomposed into  $m/p$ subsets $C_1, C_2, \cdots, C_{m/p}$ such that
	 each $C_j$ consists of $p$ points and $$\sum_{c \in C_j} e^{2 \pi i c p^{-i_n}} =0.$$
	 Then, by Lemma \ref{permu}, we have that 
	 \begin{align}\label{ceq1}
	 	c\equiv c^{\prime} +rp^{i_n-1}\ (\!\!\!\!\mod p^{i_n}) \quad \hbox { for some $r\in \{0,1,\cdots, p-1\} $  },
	 \end{align}
	  if  $c$ and  $ c^{\prime}$ lie in the same $C_j$.
	 %for some $j \in \{1,2,\cdots,m/p\}$.
	 %Write $c, c^{\prime}\in C$ in their  $p$-adic  expansions
	 %$$c=\sum_{i=0}^{\gamma}t_i p^i , \quad c'=\sum_{i=0}^{\gamma}t_{i}^{\prime} p^i.$$
	 %The equality (\ref{ceq1}) is equivalent to that $t_i=t_{i}^{\prime}$  for $0\leq i \leq i_{n}-2$.

	 %$$\chi(\frac{c}{p^{i_{n-1}}})=\chi(\frac{c^{\prime}}{p^{i_{n-1}}})$$
	 %if $c, c^{\prime}$ lie in the same $C_i$.
	 Now we consider the  equality (\ref{keyeq}) when $k=n-1$. Since $i_{n-1}< i_{n} $,   the equality (\ref{ceq1}) implies the function  $$c\mapsto e^{2 \pi i c p^{-i_{n-1}}}$$ is constant on each $C_j$.
	 For each $c\in C$, denote by $\widetilde{c}$ the point in $\{0,1,2,\cdots, p^{i_{n-1}}-1 \}$ such that 
	 $$\widetilde{c} \equiv c \ (\!\!\!\!\mod p^{i_{n-1}}).$$
	 Observe that $e^{2 \pi i c p^{-i_{n-1}}}=e^{2 \pi i \widetilde{c} p^{-i_{n-1}}}$ and that  $\widetilde{c}=\widetilde{c^{\prime}}$ if $c$ and $c^{\prime}$ lie in  the same $C_j$.
	%For each $C_j$, we take one point $c_j\in C_j$ as a representative.
	  Let $\widetilde{C}=C_{\!\!\!\! \mod{p^{i_{n-1}}}}$ be the set of all these $\widetilde{c}$.
	 So  the quality (\ref{keyeq}) with $k=n-1$ is equivalent to 
	 $$
	 \sum_{\widetilde{c} \in \widetilde{C}} e^{2 \pi i \widetilde{c} p^{-i_{n-1}}} =0.
	 $$
	  This equivalence follows from the facts that each $C_j$ contains the same number of elements.

	 Similarly, by Lemma \ref{C}, we have
	 $p\mid \frac{m}{p}$ (i.e. $p^2 \mid m$) and $\widetilde{C}$ can be decomposed into  $m/p^2$ subsets $\widetilde{C}_1, \widetilde{C}_2, \cdots, \widetilde{C}_{m/p^2}$  such that each subset  consists of $p$ elements and $$\sum_{\widetilde{c}\in \widetilde{C}_i}e^{2 \pi i \widetilde{c} p^{-i_{n-1}}} =0.$$
	 %If $c, c^{\prime}$ lie in a same group, apply the same argument,
	 %we get
	 %$$c\equiv c^{\prime} (\!\!\!\!\mod p^{i_{n-1}-1}).$$
	 By Lemma \ref{permu}, we get that 
	 \begin{align}\label{ceq2}
	 	\widetilde{c} \equiv  \widetilde{c^{\prime}} +rp^{i_{n-1}-1}\ (\!\!\!\!\mod p^{i_{n-1}}) \quad \hbox { for some $r\in \{0,1,\cdots, p-1\}$  },
	 \end{align} if  $\widetilde{c}$ and  $ \widetilde{c^{\prime}}$ lie in same $\widetilde{C}_j$.
	 
	  By induction,  we get $p^n\mid  m$. By the hypotheses $m\le p^n$, we finally get $p^n=m$.

	Furthermore,   the above argument implies that  $\mathcal{T}_{C}$ is a $p$-homogeneous tree  of $\mathcal{T}_{I, J}$-form with $I =\{i_1-1,  i_2-1, \cdots, i_n-1\}$ and $J=\{0,1,\cdots,\gamma-1\}\setminus I$. %So the compact set $\Omega$ is $p$-homogeneous.
\end{proof}

\subsection{Compact open tiles %and Dirac spectral measures 
in $\Qp$ %and in $\Z/p^\gamma\Z$
} 
Recall that $\{x\}$ denotes the  fractional part of $x\in \Q_p$.
Let $$\mathbb{L}:=\left\{\{x\}, x\in \Q_p\right\},$$ which is a complete set of representatives of the cosets of the additive subgroup $\mathbb{Z}_p$. Then $\mathbb{L}$ identified with $(\Qp/\Zp,+)$ has a structure of group with the addition defined by
$$
\{x\}+\{y\}:=\{x+y\}, \quad \forall x,y\in \Q_p.
$$
Notice that $\mathbb{L}$ is not a subgroup of $\Qp$.
Notice that $\mathbb{L}$ is the set of $p$-adic rational numbers
$$
  \sum_{i=-n}^{-1}a_ip^i  \qquad ( n \ge 1;   0\le a_i\le p-1).
$$

Let $A, B$ and $C$ be three subsets of some Abelian group. We say that $A$ is the {\em direct  sum} of $B$ and $C$ if for each $a\in A$, there exist a unique pair $(b,c)\in B\times C$  such that $a=b+c$. Then we write $A=B\oplus C$.

It is obvious that $\Qp=\Zp \oplus \mathbb{L}$, which implies that $\mathbb{L}$ is a tiling complement of $\Z_p$  in $\Q_p$.
For each integer $\gamma$, let $$\mathbb{L}_\gamma:=p^{-\gamma}\mathbb{L}.$$ 
Notice that
$$
\Qp=p^{-\gamma}\Zp \oplus p^{-\gamma} \mathbb{L}=B(0,p^\gamma)\oplus \mathbb{L}_\gamma.
$$
So $\mathbb{L}_\gamma$ is  a  tiling complement of $B(0,p^\gamma)$.
%there is a group isomorphism $(L_n,+)\simeq (\mathbb{Z}/{p^n}\mathbb{Z},+)$ and  where $|\ell|_p >1$ for any $\ell \in L\setminus \{0\}$.

 %tiles $\Qp$ if and only if the centers of balls tiles a suitable finite group $L_n$.
%Here we compare the tiling of the continuous group $\Qp$ and the tiling of the finite group
%$L_n$.
For a positive integer $\gamma$, let $C$ be a  subset of $\Z/p^\gamma\Z \simeq \{0, 1, \cdots, p^{\gamma}-1\}$. 
Let   $\Omega=\bigsqcup_{c\in C}c+ p^\gamma \Zp$,  where  $C$ is considered  as a subset of $\Z_p$. 
%We should remark  that the ring $\Z/p^{\gamma} \Z$ is occasionally considered as a subset of $\Q_p$ in our paper.
The following lemma characterize  a finite union of balls which  tiles $\Z_p$.
\begin{lemma}
The above set  $\Omega$ tiles $\Z_p$   if and only if $C$ tiles $\Z/p^{\gamma}\Z$. 
\end{lemma}
\begin{proof}
Assume that  $C$ tiles $\Z/p^\gamma\Z$, i.e. $\Z/p^\gamma\Z = C\oplus T$ for some $T\subset \Z/p^\gamma\Z$. 
One can check  that $\Z_p=\Omega \oplus T$, which implies  that $\Omega$ tiles $\Z_p$ with tile complement $T$.

Assume that $\Omega$ tiles $\Z_p$ with tiling complement $T$.  
Set $T^* = T_{\mod p^{\gamma}}$. One can check that $\Z/p^{\gamma}\Z=C\oplus T^*$.  So  $C$ tiles $\Z/p^{\gamma}\Z$ with tiling complement $T^* $.
%Now we are going to show that  $\Z/p^{\gamma}\Z$ 
%%there exists a set $S\subset \mathbb{Q}_p$ such that $S$ is a tiling complement of $C+B(0,p^{-\gamma})$ in $\mathbb{Q}_p$. Then $S+C+B(0,p^{-\gamma})=\mathbb{Q}_p$.
%Let $S^* := S_{\mod p^{\gamma}}$, a subset of the group $\mathbb{L}_{-\gamma}$. Then
%$S^* + C =\mathbb{L}_{-\gamma}$. We claim that it is a direct sum decomposition of $\mathbb{L}_{-\gamma}$:
%$$
%S^*\oplus C=\mathbb{L}_{-\gamma}.
%$$
\end{proof}
Notice that for each $a\in \Q_p$, either  $\Omega+a \subset \Z_p$ or $(\Omega+a)\cap \Z_p= \emptyset$. Then  $\Omega$ tiles $\Z_p$ if and only if   it  tiles  $\Q_p$.  So  we immediately   have the  following  
corollary.
\begin{corollary}\label{c}
The set $C$ tiles $\Z/p^{\gamma}\Z$  if and only if $\Omega$ tiles $\Q_p$.
\end{corollary}

 \subsection{$p$-homogeneous discrete set in $\Q_p$}

 Let $E$ be a discrete subset in $\Q_p$.
Recall that    $$I_{E}=\{i \in \Z:  \exists \  x,y \in E \text { such that } v_p(x-y)=i \}.$$   

The following lemma gives the relation between the 
number of elements  and possible distances in a finite subset of $\Q_p$.
 \begin{lemma}\label{estimate}
% For a positive integer $\gamma$, let  $\Lambda$ be a  nonempty finite subset of $B(0,p^{\gamma})$ such that
%$|\lambda-\lambda^{\prime}|_p>1$ for different $\lambda$ and $\lambda^{\prime}$ in $\Lambda$. Denote by
%$$Dist(\Lambda):=\{|\lambda-\lambda^{\prime}|_p: \lambda,\lambda^{\prime}\in \Lambda \hbox{ and } \lambda\neq \lambda^{\prime}\}$$
%the set of possible distances of different points in $\Lambda$. 
Let $\Lambda$ be finite subsets of $\Q_p$.
Then
$$ \sharp E \leq p^{\sharp I_{E}}.$$
\end{lemma}
\begin{proof}
Assume that $\sharp I_{E}= n$ and $I_{E}=\{i_1, i_2,\cdots,  i_n\}$ with $ i_1<i_2<\cdots<i_n$. By assumption, $E$ is contained in a ball of radius $ p^{-i_1}$. Each ball of radius  $ p^{-i_1}$ consists of $p$ balls
of radius $p^{-i_1-1}$. So we can decompose $E$ into at most $p$ subsets $E_0,E_1, \cdots, E_{p-1}$
such that
$$|\lambda-\lambda^{\prime}|_p \left\{   \begin{array}{ll}
         =p^{-i_1},& \text{if $\lambda$ and $\lambda^{\prime}$ lie in different $E_i,E_j$}.\\
        <p^{-i_1},& \text{if $\lambda$ and $\lambda^{\prime}$ lie in the same $E_i $}.
       \end{array}\right.$$
By assumption, for each $E_i$, we have $I_{E_i}\subset \{i_2, i_3 \cdots,  i_n\}$.
 We apply the above argument again, with $E$ replaced by each $E_i$.
 By induction, it suffiices to prove the conclusion when $\sharp I_E= 1$.
 Obviously, $\sharp E \leq p$ if $\sharp I_E= 1$, which completes the proof.

% loss of generality, assume that $0\in \Lambda$, since
%$$Dist(\Lambda)= Dist(a+\Lambda)
%for all $a\in \Q_p$.
%Observe that $$I^{-}:= -\log_p(Dist(\Lambda))\subset \{-\gamma, -\gamma+1, \cdots , -1\}.$$
%Let $$\Lambda^{\prime}=\left\{ \sum_{i\in I^{-}}\lambda_ip^i, \lambda_i \in \{0,1,\cdots, p-1\}\right\}.$$
%The assumption $|\lambda-\lambda^{\prime}|_p>1$ if  $\lambda\neq \lambda^{\prime}$  implies that
%$$\sharp(\Lambda)=\sharp(\Lambda(\!\!\!\!\mod \Zp)).$$
%Now we claim that
%for any  $\lambda \in \Lambda$, there exists  a $\lambda^{\prime}\in \Lambda^{\prime}$
%such that $$|\lambda-\lambda^{\prime}|\leq 1.$$
%This implies that $$\sharp \Lambda\leq \sharp \Lambda^{\prime}=p^{\sharp I^{-}}.$$

\end{proof}

Remark that a  subset  $E$  of $\Q_p$ is uniformly discrete  if $I_{E}$ is bounded from  above. Denote $  \gamma_E$ by  the maximum of $I_{E}.$
For each integer $n$,  set  $I_{E}^{\geq n}:=\{i \in  I_{E}: i\geq  n \}$. 
By Lemma \ref{estimate}, for each ball $B(a,p^{-n})$ with $n\leq \gamma_{E}$, we have 
$$\sharp(E\cap B(a,p^{-n}))\leq p^{\sharp{I_{E}^{\geq n}}}.$$
We say a   discrete set $E$ is  {\em $p$-homogeneous}   if 
$$\sharp(E\cap B(a,p^{-n}))= p^{\sharp{I_{E}^{\geq n}}}\text{ or } 0,$$
for  all integers  $n$ and all $a\in\Q_p$.  
By definition, the following lemma is immediately  obtained.
\begin{lemma}\label{2.13}
A finite set $\Lambda\subset \Q_p$ is $p$-homogeneous if and only if 
$\sharp \Lambda= p^{\sharp I_{\Lambda}}$.
\end{lemma}

%that is to say $E$ has the form that
%$$
%E\cap B(a,p^n)=\{a+\sum_{i\ge -n, i\in I_{E}} \alpha_i p^i + \sum_{j\notin I_{E}} \beta_j p^j,   \}
%$$

The following lemma shows that the $p$-homogeneous  discrete sets, under isometric transformations, admit canonical forms.

%\begin{lemma}\label{iso1}
%	Let $f:\Q_p\to \Q_p$ be an isometric transformation. If  $E$ is a $p$-homogeneous discrete  subset of $\Q_p$, then $f(E)$ is also a $p$-homogeneous discrete set with admissible $p$-order set $I_{f(E)}=I_E$.
%\end{lemma}
%\begin{proof}
%%Obviously, $I_{f(E)}=I_E$. 
%%To prove that  $f(E)$ is  a $p$-homogeneous discrete, it is sufficient to show that $$f\sharp(f(E)\cap B(a,p^{-n}))= p^{\sharp{I_{E}^{\geq n}}}\text{ or } 0,$$
%%for  any integers  $n\leq \gamma_{E}$ and any  $a\in\Q_p$.
%%Observe that $f: \Q_p\to \Q_p$ is an isometry bijection.
% For each integer $n$,  the map $$f: E\cap B(f^{-1}(a),p^{-n}) \to f(E)\cap B(a,p^n)$$ is a bijection.  Then 
%$$f(E)\cap B(a,p^{-n})=\sharp(E\cap B(f^{-1}(a),p^{-n}))=p^{\sharp{I_{E}^{\geq n}}}\text{ or } 0,$$ which implies $f(E)$ is $p$-homogeneous discrete.
%
%%Since $f$ is an isometry, it is automatically injective. Thus $f: E\to f(E)$ is an isometry bijection, and more precisely $f: E\cap B(a,p^n) \to f(E)\cap B(f(a),p^n)$ is an isometry bijection for each integer $n\geq -\gamma_{E}$. Therefore we have that $f(E)$ is a $p$-homogeneous discrete set with admissible $p$-order set $I_{f(E)}=I_E$.
%\end{proof}
\begin{lemma}\label{iso2}
Let $E$  be a  $p$-homogeneous discrete  subset of $\Q_p$.
Then there exists an isometric  transformation $f:\Q_p\to\Q_p $, such 
that $$f:E\to \widehat{E}:=\left\{\sum_{i\in I_{E}} \beta_i p^i\in \Qp: ~\beta_i\in \{0,1,2,\dots,p-1 \}   \right \}.$$
\end{lemma}

\begin{proof}

Without loss of generality, we assume that  $E$ contains $0$. Otherwise, we take a translate   $f_a(x)=x-a$ with some  $a\in E$. So $f_{a}(E)$ contains $0$. 

Recall that  $I_E$ is bounded from above  and $  \gamma_E$ is   the maximum of $I_{E}.$
For integers $n>\gamma_{E}$,  $0$ is the unique  point of $E$ which lies in the balls $p^{n}\Z_p$. Now we are going to construct an isometric transformation on $\Q_p$ by induction.

{\em Step I:} Let $n_0=\gamma_{E}$. Then the set $E \cap p^{n_0}\Z_p$ consists of  $p$ points ${x_0=0, x_1, x_2, \cdots,x_{p-1}}$ such that $ x_j\in j p^{n_0}+p^{n_0+1}\Z_p$ for  $0\leq j\leq p-1$. 
  Define 
$$f_{n_0}(x):=x-x_j+j p^{n_0} \hbox{  if  $x\in j p^{n_0}+p^{n_0+1}\Z_p$ }.$$
 So we obtain an  isometric  map $f_{n_0}: p^{n_0}\Z_p\to   p^{n_0}\Z_p$ such that  $$f_{n_0}(x_j)=j p^{n_0}  \hbox{ for all }  j\in \{0,1,,\cdots,p-1\}.$$
 %It is obvious  that $f_{n_0}(E \cap p^{n_0}\Z_p)= \widehat{E}\cap p^{n_0}\Z_p$.

{\em Step II:} Let $n_1=\gamma_{E}-1$.  We distinguish two cases: $$n_1\in I_E\hbox{ or } n_1\notin I_E. $$

If $n_1\in I_{E}$, we decompose $p^{n_1}\Z_p$ as 
$$p^{n_1}\Z_p=\bigsqcup_{j=0}^{p-1}j p^{n_1}+p^{n_0}\Z_p.$$
Applying the similar  argument as the {\em Step I}  to each  $jp^{n_1}+p^{n_0}\Z_p, 0\leq j \leq p-1 $, we obtain a isometric transformation $ g_j$ on  $j p^{n_1}+p^{n_0}\Z_p$ such that  $g_{j}(E \cap (j p^{n_1}+p^{n_0}\Z_p))= \widehat{E}\cap (j p^{n_1}+p^{n_0}\Z_p)$. So we obtain an isometric transformation $f_{n_1}$ on $p^{n_1}\Z_p$ such  that $$f_{n_1}(E \cap p^{n_1}\Z_p)= \widehat{E}\cap p^{n_1}\Z_p.$$

If $n_1\notin I_{E}$, we define 
$$ f_{n_1}(x)=\begin{cases}
f_{n_0}(x), \hbox{ if } x\in p^{n_0}\Z_p\\
x,  \hbox{ if } i\in p^{n_1}\Z_p\setminus p^{n_0}\Z_p
\end{cases}.$$ So  $f_{n_1}$  is an isometric transformation on $p^{n_1}\Z_p$ such  that $$f_{n_1}(E \cap p^{n_1}\Z_p)= \widehat{E}\cap p^{n_1}\Z_p.$$

By induction,  we  obtain an isometric transformation $f:\Q_p\to \Q_p$ such that $f(E)=\widehat{E}$.
%Let $g: \Z_p\to \Z_p (\alpha= \sum_{i=-n}^{+\infty} \alpha_i p^i \mapsto \beta =\sum_{i=-n}^{+\infty} \beta_i p^i) $ where 
%$$\beta_i=
%\begin{cases}
%\alpha_i, i\in I_{E}\\
%a_i, i\notin I_{E}\\
%
%0, ~else
%
%\end{cases}
%$$
%
%Let $\widehat{E}:=\{\sum_{i\in I_{E}} \beta_i p^i\in \Qp, \beta_i\in \{0,1,2,\dots,p-1 \}   \}$, then $\widehat{E}$ is a $p$-homogeneous discrete set with admissible $p$-order set $I_{\widehat{E}}=I_{E}=I_{E^{\prime}}$.
%
%Let $g: E\to \widehat{E}, \sum_{i=-n}^{+\infty} \alpha_i p^i \mapsto \sum_{i=-n}^{+\infty} \beta_i p^i, $ where 
%$\beta_i=
%\begin{cases}
%\alpha_i, i\in I_{E}\\
%0, ~else
%\end{cases}
%.$
%By the definition of $I_E$, for any $x,x'\in E$, we have $-\log_p|x-x'|_p\in I_E$. Hence for any $x=\sum_{i=-n}^{+\infty} \alpha_i p^i,x'=\sum_{i=-n}^{+\infty} \beta_i p^i\in E$, we have $|g(x)-g(x')|_p=\min_{i\in I_E}\{p^{-i} | \alpha_i\ne \beta_i \}=|x-x'|_p$, i.e. $g$ is isometric. By the definition of the $p$-homogeneous discrete subset, it ensures that $g$ is surjective. Since an isometry is sufficient for injection, the function $g$ is an isometric bijection. With the same construction, $g':E'\to \widehat{E}$ is an isometric isomorphism. Therefore $f=g'^{-1}\circ g$ satisfies the requirement.

\end{proof}

 \begin{proposition}\label{iso3}
  Let $E$ and $E^{\prime}$ be two $p$-homogeneous discrete sets in $\Q_p$.  Then 
  $I_{E}=I_{E^{\prime}}$ if and only if there exists an isometric transformation  $f:\Q_p\to \Q_p$ such that $f(E)=E^{\prime}$.
 \end{proposition}
 \begin{proof}
 	The `if' part of the statement is obvious. 
	
	We are going to prove the `only if ' part.
 	 We claim that  the isometric transformation constructed in Lemma  \ref{iso2}  is  a  bijection. Actually, any isometric transformation  of $\Q_p$ is surjective, which can be deduced from the fact that isometric transformations on compact metric spaces are surjective and  $\Q_p=\bigcup_{n \geq 0}p^{-n}\Zp$. Thus, by Lemma \ref{iso2}, we have two isometric bijections $f_1,f_2:\Qp\to \Qp$ such that 
 	$$
 	f_1(E)=f_2(E')=\widehat{E}=\left\{\sum_{i\in I_{E}} \beta_i p^i\in \Qp: ~\beta_i\in \{0,1,2,\dots,p-1 \}   \right \},
 	$$
 	since $I_{E}=I_{E^{\prime}}$.  Therefore,
 	$ f_{2}^{-1}\circ f_1$
 	is an  isometric transformation  of $\Q_p$, which maps $E$ onto $E^{\prime}$.
 \end{proof}

\section{Compact open spectral sets in $\mathbb{Q}_p$}
This section is devoted to the proof of Theorem 1.1.

Let $\Omega$ be a compact open set in $\mathbb{Q}_p$. 
%It is clear that if
%$\Omega$ is a spectral set, then so are its translates $\Omega + a$ ($a\in \mathbb{Q}_p$)
%and its dilations $r \Omega $ ($r \in \mathbb{Q}_p^*$). It is also true that the translation and the
%dilation don't change the tiling property and the homogeneity.
Therefore, without loss of generality,
we  assume that $\Omega$ is contained in $\mathbb{Z}_p$ and $0\in\Omega$. 
Let $\Omega$ be of the form $$\Omega=\bigsqcup_{c\in C} (c+ p^\gamma \mathbb{Z}_p),$$
where $\gamma \ge 1$ is an integer and $C\subset \{0,1,\cdots, p^{\gamma}-1\}$.
%Otherwise,
%we take any $a\in \Omega$ and  consider $p^{v} (\Omega -a)$ where
%$v = \min_{x\in \Omega} v_p(x-a)$, instead of $\Omega$.

\subsection{Homogeneity implies spectral property} \label{HomogeneitySpectral}
Assume that $\Omega$ is a $p$-homogeneous compact open set contained in
$\mathbb{Z}_p$ and containing $0$. 
We are going to show that $\Omega$
is a spectral set by constructing a spectrum for $\Omega$. 
%Let $\Omega$ be of the form $\sqcup_{c\in C} (c+ p^\gamma \mathbb{Z}_p)$
%where $\gamma \ge 1$ is an integer and $C$ is a finite set of $\mathbb{Z}_p$.
%Then the subtree associated to $\Omega$ contains all the balls $c + p^\gamma \mathbb{Z}_p$
%($c \in C$) and their ancestors.  
Let $I_{\Omega}$ be the  structure  set of $\Omega$. Then  $I_\Omega$ determines a finite $p$-homogeneous tree of type 
$\mathcal{T}_{I,J}$ with $I=I_{\Omega}\cap \{0,1,\cdots, \gamma-1\}$ and $J= \{0,1,\cdots, \gamma-1\}\setminus I$. 
 Define 
  $$
  \Lambda = (\sum_{i\in I} \Z/p\Z \cdot p^{-i-1}) + \mathbb{L}_\gamma.
  $$
  %$$\Lambda=\left\{\sum _{i=0}^{n}a_i p^{-i-1}\in \Q_p :  a_i=0 \hbox{ if } i \in I_{\Omega} \right\}\subset \mathbb{L}.$$
    We claim that  $(\Omega, \Lambda)$ is a  spectral pair.
To prove the claim, it suffices to check the equality (\ref{equ}) in Lemma \ref{Thm-SpectralMeasure}. The term on the left hand side of  the equality (\ref{equ})
is equal to
%By Lemma \ref{fourierlem},  we have
%$$\widehat{1_{\Omega}}(\xi)=\frac{1}{p^\gamma}\sum_{c\in C_{I,J}}\chi(-c\xi)\cdot 1_{B(0,p^\gamma)}(\xi).$$
\begin{align}\label{cc}
 \sum_{\lambda \in \Lambda} |\widehat{1_{\Omega}}(\lambda-\xi)|^2
&=\frac{1}{p^{2\gamma}}\sum_{\lambda \in \Lambda}1_{B(\xi,p^\gamma)}(\lambda)\sum_{c,c^{\prime}\in C}\chi((c-c^{\prime})(\lambda-\xi))\nonumber \\
&=\frac{1}{p^{2\gamma}}\sum_{\lambda \in \Lambda\cap B(\xi,p^\gamma) }\ \ \sum_{c,c^{\prime}\in C}\chi((c-c^{\prime})(\lambda-\xi)).
\end{align}
%&=\frac{1}{p^{2\gamma}}\sum_{\lambda \in \Lambda}1_{B(\xi,p^\gamma)}(\lambda)(p^{\sharp I}+\sum_{\substack{c,c^{\prime}\in C_{I,J}\\ c\neq c^{\prime}}}\chi((c-c^{\prime})(\lambda-\xi))).
%\end{align*}
 Let $\xi =\sum_{i=v_p(\xi)}^{\infty} \xi_{i}p^i \in \Qp$.  We set $\xi_{i}=0$ if $i< v_p(\xi)$,
 so that $\xi =\sum_{i=-\infty}^{\infty} \xi_{i}p^i$. Let 
 
 $$
 \xi_{\star}=\sum_{i=-\gamma }^{-1}\xi_{i}p^i,\quad \xi^{\prime}=\sum_{j=v_p(\xi)}^{-\gamma-1}\xi_{j}p^j. 
 $$
Remark   that $\xi^{\prime}$ is $0$ when $ v_p(\xi)>-\gamma-1$.
  Then we have    $\{\xi\}=\xi_{\star}+\xi^{\prime}$ and   $|\xi-\xi^{\prime}|_p \leq p^\gamma$  which implies  $B(\xi,p^\gamma)= B(\xi^{\prime},p^\gamma)$. 
 For   $\lambda=\sum _{i=0}^{n}a_i p^{-i-1} \in \Lambda$,  observe that  $|\lambda-\xi|_p\leq p^\gamma$ if and only if
 $a_i= \xi_{-i-1}$  for all $i\geq \gamma$.
So we get
\begin{align*}
\Lambda \cap B(\xi,p^\gamma) &=\xi'+\sum_{i\in I} \Z/p\Z \cdot p^{-i-1},
 \end{align*}
 which consists of $p^{\sharp I} $ elements. Using this last fact, the fact $|\Omega|^2=p^{-2(\gamma-\sharp I)}$
 and the equality (\ref{cc}), to prove the equality (\ref{equ}), we have only to prove that
 \begin{equation}\label{cc2}
 \sum_{\lambda \in \Lambda \cap B(\xi,p^\gamma) } \chi((c-c^{\prime})(\lambda-\xi))=0  \quad \hbox{ for } c\neq c^{\prime}.
 \end{equation}
 The possible distances between $c$ and $c^{\prime}$ are of the form $p^{-i}$ with $i\in I$.
 Fix two different $c$ and $c^{\prime}$ in $C$.   Write $$c-c^{\prime}=p^{i_0}s,$$ for some $i_0\in I$ and some $s\in\Zp^{\times}$.
 Set $I_{i_0}=I\cap[i_0,\gamma-1]$.
 For  any $\lambda =\sum_{i\in I}a_ip^{-i-1} +\xi^{\prime}\in \Lambda\cap B(\xi, p^\gamma)$, we have
 \begin{align*}
 (c-c^{\prime})(\lambda-\xi)&\equiv (c-c^{\prime}) (\sum_{i\in I}a_ip^{-i-1} -\xi_\star)\quad (\!\!\!\!\mod \Zp)\\
 &\equiv -\xi_\star (c-c^{\prime}) + \frac{ s  \sum_{i\in I_{i_0}}a_i p^{\gamma-i-1}}{p^{\gamma-i_0}} \quad (\!\!\!\!\mod \Zp)
 \end{align*}
so that
 $$ \chi((c-c^{\prime})(\lambda-\xi))=\chi(-\xi_\star (c-c^{\prime}))  \prod_{i\in I_{i_0}}\chi\Big(\frac{ s  a_i}{p^{i-i_0+1}}\Big).$$
 From this, we observe that as function of $\lambda$,
$ \chi((c-c^{\prime})(\lambda-\xi))$ only depends on the coordinates
$a_i$ of $\lambda$ with $i\in I_{i_0}$. Then,
  %= \chi((c-c^{\prime})(\lambda^{\prime}-\xi))$ if $|\lambda-\lambda^{\prime}|\leq p^{i_0}$.
  by the definition of $\Lambda$,
  %$\Lambda \cap B(\xi,p^\gamma)$ can be decomposed as $p^{\sharp(I_{i_0})}$ subsets  such that each %subset contains $p^{\sharp(I\setminus I_{i_0})}$ elements and $|\lambda-\lambda^{\prime}|\leq %p^{i_0}$ if $\lambda$ and $\lambda^{\prime}$ lie in a  same  subset.
  for each  $\lambda =\sum_{i\in I_0}a_ip^{-i-1} +\xi^{\prime}\in \Lambda\cap B(\xi, p^\gamma)$, there are  $p^{\sharp(I\setminus I_{i_0})}$ points  $\lambda'\in \Lambda \cap B(\xi,p^\gamma)$ such that $\chi((c-c^{\prime})(\lambda-\xi)) = \chi((c-c^{\prime})(\lambda^{\prime}-\xi))$.
 So we get
 \begin{align*}
 \sum_{\lambda \in \Lambda \cap B(\xi,p^\gamma) } \chi((c-c^{\prime})(\lambda-\xi))%\\
 %&=\sum_{\substack {i\in I\\ a_i\in \{0,1,\cdot,p-1\}}} \chi(-\xi_\gamma)\chi(p^{i_0}\cdot  s \cdot \sum_{i\in I}a_ip^{-i-1})\\
 %&=\chi(-\xi_\gamma)\sum_{\substack {i\in I\\ a_i\in \{0,1,\cdots,p-1\}}} \chi(\frac{p^{i_0}\cdot  s \cdot \sum_{i\in I}a_i p^{\gamma-i-1}}{p^\gamma})\\
 %&=\chi(-\xi_\gamma (c-c^{\prime}))\sum_{\substack {i\in I\\ a_i\in \{0,1,\cdots,p-1\}}}\chi(\frac{ s \cdot \sum_{i\in I}a_i p^{\gamma-i-1}}{p^{\gamma-i_0}}).\\
 = p^{\sharp(I\setminus I_{i_0})} \chi(-\xi_\star(c-c^{\prime}))\prod_{i \in I_{i_0}} \sum_{a_i=0}^{p-1} \chi\Big(\frac{ s  a_i}{p^{i-i_0+1}}\Big).
 %\sum_{\substack {i\in I_{i_0}\\ a_i\in \{0,1,\cdots,p-1\}}}\chi(\frac{ s \cdot \sum_{i\in I_{i_0}}a_i p^{\gamma-i-1}}{p^{\gamma-i_0}}).
 % &=\chi(-\xi_\gamma)\cdot p^{\sharp(I\setminus I_{i_0})}\sum_{\substack {i\in I_{i_0}\setminus \{i_0\}\\ a_i\in \{0,1,\cdot,p-1\}}} \sum_{a_{i_0}\in \{0, 1, \cdots, p-1 \}}\chi(\frac{ s \cdot \sum_{i\in I_{i_0}}a_i p^{\gamma-i-1}}{p^{\gamma-i_0}})\\
  \end{align*}
  Therefore, we shall prove (\ref{cc2}) if we prove that the factor corresponding to
    $i=i_0$ on the right hand side of the last equality is zero, i.e.
    \begin{align}\label{zero}
\sum_{a_{i_0}=0}^{p-1}\chi\Big(\frac{  s a_{i_0} }{p}\Big)=0.
\end{align}
This is really true because of Lemma \ref{multi} and
$$\sum_{a_{i_0}=0}^{p-1}\chi\Big(\frac{  a_{i_0} }{p}\Big)=0.$$
%Lemma \ref{multi} implies that
%\begin{align}\label{zero}
%\sum_{a_{i_0}=0}^{p-1}\chi(\frac{  s a_{i_0} p^{\gamma-i_0-1}}{p^{\gamma-i_0}})=0
%\end{align}
% So we have
%$$\prod_{i \in I_{i_0}} \sum_{a_i=0}^{p-1} \chi (\frac {s a_i p ^{\gamma -i %-1}}{p^{\gamma-i_0}})=0$$
%Hence, we obtain that $$\sum_{\lambda \in \Lambda \cap B(\xi,p^\gamma) } \chi((c-c^{\prime})(\lambda-\xi))=0 \hbox{ if } c\neq c^{\prime},$$ which
Thus we have proved that $\Omega$ is a spectral set.
%complete the proof of the ``if " part.

\subsection{Spectral property implies homogeneity}\label{spectraltohomo}
%Let   $\Omega$ be a compact open spectral set contained in  $\mathbb{Z}_p$. It can be represented as $\Omega=\bigsqcup_{c\in C}c+ p^\gamma \Zp$, where $\gamma$ is a nonnegative integer and  $C\subset  \{0,1, \cdots,  p^{\gamma}-1\}$. 

Assume that $\Lambda$ is  a spectrum of $\Omega$. We are going to show that $\Omega$ is $p$-homogeneous.

 By Lemma \ref{number},  we have  $\sharp(B(0,{p^\gamma})\cap \Lambda)= \sharp C$.
 For simplicity, let $\sharp C= m$.
Set $$D=\left\{|\lambda-\lambda^{\prime}|_p: \lambda,\lambda^{\prime}\in B(0,{p^\gamma})\cap \Lambda  \hbox{ and }  \lambda \neq \lambda^{\prime}\right\}$$
be the set of possible distances of different spectrum points in the ball $B(0,{p^\gamma})$.  Notice that $\log_p(D)\subset \{1,2, \cdots, \gamma\}$.
Assume that $\sharp D= n$ and $$\log_p(D)=\{i_1,i_2,\cdots, i_n\}\  \hbox{ with}\ \
1\leq i_1<i_2<\cdots<i_n\leq \gamma.$$
Observe that
\begin{align*}
\langle \chi_{\lambda}, \chi_{\lambda^\prime} \rangle %&=\int_{\Omega}\chi_{\lambda-\lambda^{\prime}}(\xi)d\xi \\
%&= \widehat{1_{\Omega}}(\lambda-\lambda^{\prime}). \\
&=\frac{1}{p^{\gamma}}1_{B(0,p^{\gamma})}(\lambda-\lambda^{\prime})\sum_{c\in C}\chi(-c(\lambda-\lambda^{\prime})).
\end{align*}
So, the orthogonality of $\{\chi_{\lambda}\}_{\lambda\in \Lambda}$ implies
\begin{equation}\label{OT}
\sum_{c\in C}\chi(-c(\lambda-\lambda^{\prime}))=0\qquad
(\forall \lambda, \lambda^{\prime} \in \Lambda,  0<|\lambda-\lambda'|_p\le p^{\gamma}).
\end{equation}
By (\ref{OT}) and Lemma \ref{multi}, it deduces that $C$ satisfies the conditions in Theorem \ref{keytheorem}. Therefore, $C$ is a $p$-homogeneous tree.

On the other hand,  $\sharp(B(0,{p^\gamma})\cap \Lambda)= \sharp C=p^n$. Thus, by  Lemma  \ref{2.13}, the discrete set   $B(0,{p^\gamma})\cap \Lambda$ is  $p$-homogeneous  with $I_{B(0,{p^\gamma})\cap \Lambda}=-\log_p(D)$.

\subsection{Equivalence between homogeneity and tiling}

	%Let $$\Omega:=\bigsqcup_{1\le i\le k}(t_i+B(0,1))\subset B(0,p^n),$$
%where $k,n\in \mathbb{N}_+, |t_i|>1$. 
Due to Lemma \ref{b} and  Corollary \ref{c}, it is sufficient to prove that $C$ is a tile of $\Z/p^{\gamma}\Z$ if and only if $\mathcal{T}_{C}$ is a $p$-homogeneous tree.   We shall finish the proof when we have proved the equivalence between the tiling property  and the $p$-homogeneity of a set  in $\mathbb{Z}/p^\gamma\mathbb{Z}$. This will be done in the next section.

\section{Spectral sets and tiles in $\mathbb{Z}/p^\gamma\mathbb{Z}$}
In this section, we characterize spectral sets and tiles in the finite group $\mathbb{Z}/p^\gamma\mathbb{Z}$. Spectral sets and tiles in this group are the same which are characterized by  a simple geometric
property that we qualify as $p$-homogeneity. They can also be characterized by their
Fourier transforms.

Recall that the characters of $\mathbb{Z}/p^\gamma \mathbb{Z}$ are  the functions
$$x\mapsto e^{ \frac{2\pi i kx}{p^\gamma}}, \quad k \in \mathbb{Z}/p^\gamma \mathbb{Z}.$$ 
We identify $\mathbb{Z}/p^\gamma \mathbb{Z}$ to $\{0,1, \cdots, p^{\gamma}-1\}$ which can be viewed as a subset of $\Q_p$. The restriction  of the characters 
$\chi_{\frac{k}{p^{\gamma}}}, k=0,1,2, \cdots,p^{\gamma}-1$  of $\Q_p$ on $\Z/p^{\gamma}\Z$ are exactly the characters of $\Z/p^{\gamma}\Z$.

For a subset $C$ of $\Z/p^{\gamma}\Z$ which is viewed as a subset of $\Q_p$, let  $\delta_C$ be the uniform probability measure  in $\Q_p$.  By definition,  we immediately  have the  following lemma.
\begin{lemma}\label{spectralmeasure}
Let $C, \Lambda \subset\{0,1,2,\cdots, p^{\gamma}-1\}$. Then $(C,\Lambda)$ is a spectral pair  in $\Z/p^{\gamma}\Z$ if and only if $(\delta_{C}, \frac{1}{p^{\gamma}}\Lambda)$ is a spectral pair in $\Q_p$. 
\end{lemma}

 The Fourier transform of a function $f$ defined on $\mathbb{Z}/p^\gamma \mathbb{Z}$
is defined as follows
$$
\widehat{f}(k)=\sum_{x\in  \mathbb{Z}/p^\gamma \mathbb{Z}}
f(x)e^{- \frac{2\pi i kx}{p^\gamma}}, \qquad (\forall k\in \mathbb{Z}/p^\gamma \mathbb{Z}).
$$

%\begin{proof}
%	Assume that  $C$ is a spectral set in  $\Z/p^\gamma \Z$ with spectrum $\Lambda\subset \{0, 1, \cdots, p^{\gamma}-1\}$. Then  one can check that  $\delta_C$ is a spectral measure in $\Qp$ with spectrum $\frac{1}{p^\gamma}\Lambda$.
%	
%	On the other hand, if $\delta_C$ is a spectral measure in $\Qp$ with spectrum $\Gamma$. Since $C$ is a set of integers, we can assume $\Gamma \subset \mathbb{L}$ and $0\in \Gamma$. Then we have 
%	$$
%	\sum_{c\in C} e^{2\pi i \lambda c}=0, \forall \lambda\in \Gamma \setminus \{0\}.
%	$$
%	By Lemma \ref{root}, there exists $c_1$ and $c_2$ in $C$ such that 
%	$$
%	p=|c_1\lambda-c_2\lambda|_p=|\lambda|_p\cdot|c_1-c_2|_p\ge |\lambda|_p \frac{1}{p^{\gamma-1}}.
%	$$
%	So we have $|\lambda|\le p^{\gamma}$ for any $\lambda\in \Gamma \setminus \{0\}$. Thus we can choose $p^{\gamma}\Gamma$ as a spectrum of $\delta_C$ in $\Z/p^\gamma \Z$ and therefore $\delta_C$ is a spectral measure in $\Z/p^\gamma \Z$.
%\end{proof}

\begin{theorem}\label{Z/Zp}
	Let $C\subset \mathbb{Z}/p^\gamma\mathbb{Z}$ and $\mathcal{T}_{C}$ be the associated tree. The following  are equivalent.
	\begin{itemize}
		\item[(1)] $\mathcal{T}_C$ is a $p$-homogeneous tree.
		\item[(2)] For any $ 1\leq i\leq \gamma,  \sharp(C_{\!\!\!\! \mod{p^i}})=p^{k_i}$, for some $k_i\in \mathbb{N}$.
		\item[(3)] There exists a subset $I\subset \mathbb{N}$ such that $\sharp I=\log_p(\sharp C)$ and  $\widehat{1_C}(p^{\ell}) =0$ for $\ell\in I$.
		\item[(4)] There exists a subset $I\subset \mathbb{N}$ such that $\sharp I\ge \log_p(\sharp C)$ and  $\widehat{1_C}(p^{\ell}) =0$ for $\ell\in I$.
		\item[(5)] $C$ is a tile of $\mathbb{Z}/p^\gamma \mathbb{Z}$.
		\item[(6)] $C$ is a spectral set in $\mathbb{Z}/p^\gamma\mathbb{Z}$.
	\end{itemize}
\end{theorem}
%Since we have proved the equivalence between $(1)$ and $(6)$ in the above section, now we will prove the equivalence among $(1), (2), (3), (4)$ and $(5)$.
\begin{proof}
	$(1)\Rightarrow(2)$: It follows from the  definition of $p$-homogeneous subtree.
	
	$(2)\Rightarrow(3)$: From $\sharp C=p^{k_\gamma}$ we get $\log_p(\sharp C)=k_\gamma$.  For simplicity, denote by  
	$C_j=C_{\!\!\!\!\mod p^{j}}$ for $1\le j\le \gamma$.
	
	Define
	$$
	I:=\{\gamma-j: \sharp C_{j-1}< \sharp C_{j}\}\subset \{0,1,\cdots,\gamma-1 \}, 1\le j\le \gamma.
	$$
	Then $\sharp I=k_\gamma$. For any $j$ such that $\gamma-j\in I$, we have $\sharp C_j=p \sharp  C_{j-1}$. More precisely,
	$$
	C_j=C_{j-1}+\{0,1, 2,\dots, p-1\} p^{j-1}.
	$$
	Thus
	\begin{align*}
		\widehat{1_C}(p^{\gamma-j})
		&=\sum_{t\in  C}
		e^{- \frac{2\pi i }{p^{j}}t}=p^{k_{\gamma}-k_j}\sum_{t\in  C_{j}}e^{- \frac{2\pi i }{p^{j}}t}\\
		&=p^{k_{\gamma}-k_j}\sum_{t\in  C_{j-1}}\sum_{l=0}^{p-1}
		e^{- \frac{2\pi i }{p^{j}}(t+lp^{j-1})}\\
		&=p^{k_{\gamma}-k_j} \sum_{t\in  C_{j-1}}e^{- \frac{2\pi i }{p^{j}}t}\sum_{l=0}^{p-1}
		e^{- \frac{2\pi i }{p}l}
		=0,
	\end{align*}
	i.e. $\widehat{1_C}(p^{\ell}) =0$ for $\ell\in I$.
	
	$(3)\Rightarrow(4)$: Obviously.
	
	$(4)\Rightarrow(1)$: Observe that $ \widehat{{1}_{C}}(p^\ell) =0$ means
	$$
	\sum_{t\in  C}
	e^{- \frac{2\pi i t}{p^{\gamma-\ell}}} =0,
	$$
which is exactly the condition in Theorem \ref{keytheorem}.
%	By Lemma  \ref{multi},  $ \widehat{{1}_{T}}(p^\ell) =0$ is equivalent to that $\chi_{\lambda} $ %and $\chi_{\lambda^\prime}$ are orthogonal if $|\lambda-\lambda^{\prime}|=p^{n-\ell}$. In  Section %\ref{spectraltohomo},  the set of possible distances of different spectrum points is essential for  %the proof of ``Spectral implies homogeneity".
Therefore we can prove that $\sharp I=\log_p(\sharp C)$ and $\mathcal{T}_C$ is a $p$-homogeneous tree.
	
	$(1)\Rightarrow(5)$: Assume that $\mathcal{T}_{C}$ is a $p$-homogeneous tree  $\mathcal{T}_{I,J}$. It is obvious that $C$ has the tiling property
	$C\oplus S=\mathbb{Z}/p^\gamma \mathbb{Z}$ with the tiling complement
	$$S=\left\{\sum_{i\in J}a_ip^i: a_i\in \{ 0,1,\dots,p-1 \} \right\}.$$
	
	$(5)\Rightarrow(4)$: Assume that $C$ is a tile of $\mathbb{Z}/p^\gamma \mathbb{Z}$. That is to say, there exists a set $S\subset \mathbb{Z}/p^{\gamma}\mathbb{Z}$ such that $C\oplus S=\mathbb{Z}/p^{\gamma}\mathbb{Z}$. Since $\sharp (C\oplus S)= \sharp C  \cdot \sharp S$,
	%(we can treat $p^{\gamma}L_n$ as a subset of $\mathbb{Q}_p$ under additon mod $p^{\gamma}$)
	$\sharp C$ divides $\sharp (\mathbb{Z}/p^{\gamma}\mathbb{Z})=p^{\gamma}$. The
	equality $C\oplus S=\mathbb{Z}/p^{\gamma}\mathbb{Z}$ can be rewritten as
	$$
	\forall x\in \mathbb{Z}/p^{\gamma}\mathbb{Z}, \qquad
	\sum_{y\in \mathbb{Z}/p^{\gamma}\mathbb{Z}}{1}_{C}(y){1}_{S}(x-y)=1.
	$$
	In other words,	${1}_{C}*{1}_{S} =1$, where the convolution is that in group $\Z/ p^\gamma \Z$.
	%	\sum_{x\in p^{\gamma}L_n}\mathbbm{1}_{T}*\mathbbm{1}_{S}(y)=1, \forall y\in p^{\gamma}L_n.
	%	$$
	Then we have
	$$
	\widehat{{1}_{C}}\cdot\widehat{{1}_{S}}=p^{\gamma}\delta_0,
	$$
	where $\delta_0$ is the Dirac measure concentrated at $0$.
	Consequently
	$$
	Z(\widehat{{1}_{C}})\cup Z(\widehat{{1}_{S}})=\mathbb{Z}/p^{\gamma} \mathbb{Z}\setminus \{0\},
	$$
	where $Z(\widehat{f}):=\{x: \widehat{f}(x)=0\}$ is the set of zeros of $\widehat{f}$. In particular,
	the powers $p^\ell$ with $\ell =0, 1, 2, \cdots, \gamma-1$ are zeroes of  either
	$\widehat{{1}_{C}}$ or  $\widehat{{1}_{S}}$. Let
	$$
	C_z = \left\{l\in \{0, 1, 2, \cdots, \gamma-1\}: \widehat{{1}_{C}}(p^\ell) =0\right\},
	$$
	$$
	S_z = \left\{l\in \{0, 1, 2, \cdots, \gamma-1\}: \widehat{{1}_{S}}(p^\ell) =0\right\}.
	$$
	Since $C_z\cup S_z=\{0,1,2,\dots,\gamma-1 \}$, we have $\sharp C_z+\sharp S_z\ge \gamma$.
	On the other hand, we have
	$
	\log_p\sharp C+\log_p\sharp S=\gamma.
	$
	It follows that we have
	$$
	\sharp C_z\ge \log_p\sharp C\quad \mbox{\rm or} \quad \sharp S_z\ge \log_p\sharp S.
	$$
	If $\sharp C_z\ge \log_p\sharp C$, we are done.
	If $\sharp S_z\ge \log_p\sharp S$, the arguments used in the proof  $(4)\Rightarrow(1)$ leads to $\sharp S_z=\log_p\sharp S$. So  we have $\sharp C_z\ge \log_p\sharp C$ .
	
	$(1)\Leftrightarrow (6)$: In Sections \ref{HomogeneitySpectral} and \ref{spectraltohomo},  we have proved the equivalence between (1) and that $\Omega=\bigsqcup_{c\in C}c+ p^\gamma \Zp$ is a spectral set in $\Qp$. By Lemma \ref{spectralmeasure}, we have that (6) is equivalent to that $\delta_C$ is a spectral measure in  $\Qp$. Then what we have to prove is the follwoing equivalence:
	$$
	\Omega ~\text{is a spectral set in} ~~\Qp \Leftrightarrow \delta_C~\text{is a spectral measure in}~~\Qp.
	$$
	Recall that  $\mathbb{L}_\gamma= p^{-\gamma}\mathbb{L}$. It suffices to prove that
$$
(\Omega, \Lambda_{C}+\mathbb{L}_{\gamma}) ~\text{is a spectral pair} \Leftrightarrow
(\delta_{C}, \Lambda_{C}) ~\text{is a spectral pair in}~~\Qp  
$$
where $\Lambda_C\subset B(0,p^\gamma)$ is some finite set, because it is known from Sections \ref{HomogeneitySpectral} and \ref{spectraltohomo}, that $\Omega$ has a spectrum of the form $\Lambda_{C}+\mathbb{L}_{\gamma}$ if it is a spectral set.
	By Lemma \ref{Thm-SpectralMeasure}, $(\delta_{C}, \Lambda_{C})$ is a spectral pair in $\Qp$
	if and only if
	\begin{equation}\label{EQdiscrete1}
	\forall \xi \in \Qp, \quad\sum_{\lambda\in \Lambda_{C}}\left| \frac{1}{\sharp C}\sum_{c\in C} \chi(-c(\lambda-\xi))\right|^2= 1.
	\end{equation}
	Recall that
	$$\widehat{1_{\Omega}}(\lambda-\xi)=p^{\gamma}1_{B(0,p^{\gamma})}(\lambda-\xi)\sum_{c\in C}\chi(-c(\lambda-\xi)).$$
	The equality  (\ref{EQdiscrete1}) is then equivalent to
	\begin{align*}
	\forall \xi \in \Qp, \quad &\sum_{\lambda\in \Lambda_{C}+\mathbb{L}_\gamma}| \widehat{1_{\Omega}}(\lambda-\xi)|^2\\
	&=p^{2\gamma}\sum_{\lambda\in \Lambda_{C}+\mathbb{L}_\gamma}1_{B(0,p^{\gamma})}(\lambda-\xi)|\sum_{c\in C}\chi(-c(\lambda-\xi))|^2\\
	&=p^{2\gamma}\sum_{\lambda\in \Lambda_{C}+\xi^{\prime}}|\sum_{c\in C}\chi(-c(\lambda-\xi))|^2\\
	&=p^{2\gamma}\sum_{\lambda\in \Lambda_{C}}|\sum_{c\in C}\chi(-c\lambda)|^2
	=(\sharp C)^2p^{2\gamma}=|\Omega|^2,
	\end{align*}
	which means, by Lemma \ref{Thm-SpectralMeasure}, that $(\Omega, \Lambda_{C}+\mathbb{L}_\gamma)$
	is a spectral pair.
	
\end{proof}

\section{Uniqueness of spectra and tiling complements}
In this section, we shall investigate the structure of the spectra and tiling complements of a $p$-homogeneous compact set. 
Without loss of generality, we  assume that $\Omega$ is of the form $$\Omega=\bigsqcup_{c\in C} (c+ p^\gamma \mathbb{Z}_p),$$
where $\gamma \ge 1$ is an integer and $C\subset \{0,1,\cdots, p^{\gamma}-1\}$.  We immediately get that 
$$I_{\Omega}\subset \mathbb{N} \text{ and  }  n\in  I_{\Omega} \text { if  }n\geq \gamma.$$
Assume that $\Lambda$  is a spectrum of $\Omega$ and $T$ is a tiling complement of $\Omega$.
Notice that  $\Lambda$ and $T$ are discrete  subsets of $\Q_p$ such
that $$|\lambda-\lambda^{\prime}|_p >1, \  \ \ \ \text{ if $\lambda, \lambda^{\prime}\in  \Lambda$ and $\lambda \neq \lambda^{\prime}$ }$$
and 
$$|\tau-\tau^{\prime}|_p>p^{-\gamma}, \  \ \ \ \text{ if $\tau, \tau^{\prime}\in  T$ and $\tau \neq \tau^{\prime}$ }.$$
Now we are going to characterize of the spectra and tiling components.
\begin{theorem}\label{fulldiscrete}
Let  $\Omega\subset \Q_p$ be a $p$-homogeneous compact open set  with the admissible $p$-order set $I_{\Omega}$.  \\
%Assume that  $\Lambda$  is a spectrum of $\Omega$ and $T$ is a tiling complement of $\Omega$.\\
 \indent {\rm (a)} The set $\Lambda$ is a spectrum of $\Omega$ if and only if it is   $p$-homogeneous  discrete set  with admissible $p$-order set $I_{\Lambda}=-(I_{\Omega}+1)$ .\\
  \indent {\rm (b)} The set $T$ is a  tiling complement  of $\Omega$  if and only if it is a $p$-homogeneous  discrete set with admissible $p$-order set  $I_{T}=\Z\setminus I_{\Omega} $.
\end{theorem}
\begin{proof}
Without loss of generality, assume   $\Omega=\bigsqcup_{c\in C} (c+ p^\gamma \mathbb{Z}_p)$,  where $\gamma \ge 1$ is an integer and $C\subset \{0,1,\cdots, p^{\gamma}-1\}$.  For an  integer $n$, let  $$I_{\Omega}^{\leq n}=\{i\in I_{\Omega}, i\leq n\}.$$ 

{\rm (a)} In Section \ref{spectraltohomo}, we have proved that $\Lambda\cap B(0,p^\gamma)$ is a $p$-homogeneous  discrete set with admissible $p$-order set $I_{\Lambda\cap B(0,p^\gamma)}=-(I_{\Omega}^{\leq \gamma-1} +1)$. Note that,  any integer $n\ge \gamma$, the set $\Omega$ can be written as $$\Omega=\bigsqcup_{c\in C_n} (c+ p^n \mathbb{Z}_p),$$ where
$C_n\subset \{0,1,\cdots, p^{n}-1\}$.  The same argument implies that  the finite set  $\Lambda\cap B(0,p^n)$   is  $p$-homogeneous with  $I_{\Lambda\cap B(0,p^n)}=-(I_{\Omega}^{\leq n-1} +1)$.
%By the equations (\ref{cc}) and (\ref{equ}), we have that $(\Omega, \Lambda)$ is a spectral pair if and only if for all $\xi \in \Qp$,
%$$
%|\Omega|^2=\frac{1}{p^{2\gamma}}\sum_{\lambda \in \Lambda\cap B(\xi,p^\gamma) }\  \sum_{c,c^{\prime}\in %C}\chi((c-c^{\prime})(\lambda-\xi)).
%$$
%For any $a\in \Qp$, take the same calculation of replacing $B(0,p^n)$ by $B(a,p^n)$, we can prove that $\Lambda\cap B(a,p^n)$ is a $p$-homogeneous  discrete set with admissible $p$-order set $I_{\Lambda\cap B(a,p^n)}=-(I_{\Omega}^{\leq n}+1)$ for any integer $n\ge \gamma$. 
By Lemma \ref{2.13} and the definition of $p$-homogeneity, $\Lambda$ is a $p$-homogeneous  discrete set with the admissible $p$-order set  $I_{\Lambda}=-(I_{\Omega}+1)$.

In fact,  it is routine  to  check that the equation (\ref{spectral criterion}) holds for any  $p$-homogeneous discrete set  $\Lambda$ with $I_{\Lambda}=-(I_{\Omega}+1)$.
 So  the $p$-homogeneity  of $\Lambda$ and  the equality  $I_{\Lambda}=-(I_{\Omega}+1)$  is sufficient for  $\Lambda$ being  a spectrum of $\Omega$.
 
\medskip
{\rm (b)} By  Corollary \ref{c} and Theorem \ref{Z/Zp}, $T\cap \Z_p$ is a $p$-homogeneous  discrete set with admissible $p$-order set  $$I_{T\cap \Z_p} = \{0, \dots, \gamma-1 \} \setminus I_{\Omega}^{\leq \gamma-1}. $$

Similarly, for any $a\in \Qp$,  $T\cap (a+\Z_p)$ is a $p$-homogeneous  discrete set with  $I_{T\cap ( a+\Z_p)}=I_{T\cap \Z_p} $. Since two balls of same size are either identical or disjoint, $T\cap (p^{-1}\Z_p)$ is  a $p$-homogeneous  discrete set with  $I_{T\cap ( p^{-1}\Z_p)}=I_{T\cap \Z_p}\cup\{-1\}$. 

An argument by induction shows that $T\cap (p^{-n}\Z_p)$ is $p$-homogeneous with $$I_{T\cap ( p^{-n}\Z_p)}=I_{T\cap \Z_p}\cup\{-1,-2,\cdots, -n \}.$$
As in (a),  we get that $T$ is a $p$-homogeneous  discrete set with admissible $p$-order set $I_{T}=\Z\setminus I_{\Omega} $.

On the other hand, one can check that any $p$-homogenous discrete set $T$ with $I_{T}=\Z\setminus I_{\Omega}$  is a  tiling complement of $\Omega$.
\end{proof}
%Actually,  in the proof of Theorem \cite{{fulldiscrete}   a $p$-homogenous discrete set, which could be a spectrum  or tiling complement of $\Omega$, depends only on the set 
%\begin{proposition}  Let $\Omega$ be a $p$-homogeneous compact open  set in $\Q_p$ and  let $\Lambda$,  $T$ be  $p$-homogeneous discrete sets in  $\Q_p$. We have\\
% \indent {\rm (a)} $(\Omega, \Lambda)$ is a spectral pair if and only if $I_{\Lambda}=-(I_{\Omega}+1)$ ,
%
% \indent {\rm (b)} $(\Omega, T)$ is a tiling pair  if and  only if $I_{T}=\Z\setminus I_{\Omega} $.
%
%\end{proposition}  

\begin{proof}[Proof of Theorem \ref{spetrumandtranslate}]
	In the proof of Theorem \ref{main}, we have constructed a spectrum $ \Lambda=\sum_{i\in I_\Omega} \Z/ p\Z \cdot p^{-i-1}$ for $\Omega$ and a tiling complement $ T=\sum_{i\notin I_\Omega} \Z/ p\Z \cdot p^{i} $ for $\Omega$. Therefore, this theorem is an immediate consequence of  Theorem \ref{fulldiscrete}, Lemma \ref{iso2} and Proposition \ref{iso3}.
\end{proof}

  Let us finish this section by geometrically presenting  the canonical  spectrum and  the canonical  tiling complement of a  compact open spectral set.
Assume  that    $\Omega=\bigsqcup_{c\in C} (c+ p^\gamma \mathbb{Z}_p)$  with $\gamma \ge 1$ is an integer and $C\subset \{0,1,\cdots, p^{\gamma}-1\}$. Notice that $n\in I_{\Omega}$ if $n\geq\gamma$.
 Set $\Lambda_{\gamma}=\Lambda \cap B(0,p^{\gamma})$ and $T_{0}=T \cap B(0,1)$.  Then $\Lambda= \Lambda_{\gamma}\oplus \mathbb{L}_{\gamma}$
 and  $T= T_{0}\oplus \mathbb{L}$. The sets $\Lambda_{\gamma}$ and $T_0$ are $p$-homogeneous.  If we consider $p^{\gamma}\Lambda_{\gamma}$ and $T_0$ as subsets of $\Z/p^{\gamma}\Z$, they will determine two subtrees of $\mathcal{T}^{(\gamma)}$.
 The following example show the  relations among 
 $\Omega$, $ \Lambda_{\gamma}$ and $T_0$.

%The above theorem gives a geometrical characterization of the spectra and tiling complements of a $p$-homogeneous compact set. We can also associate the spectra and tiling complements  with $p$-homogeneous tree structure. 
%See the following figure for the tree structure of a tiling complement of the compact open set of $\Omega=2+4\Z_2 \sqcup  3+4\Z_2$.
 \begin{figure}[H]
  \centering
  \includegraphics[width=1\textwidth]{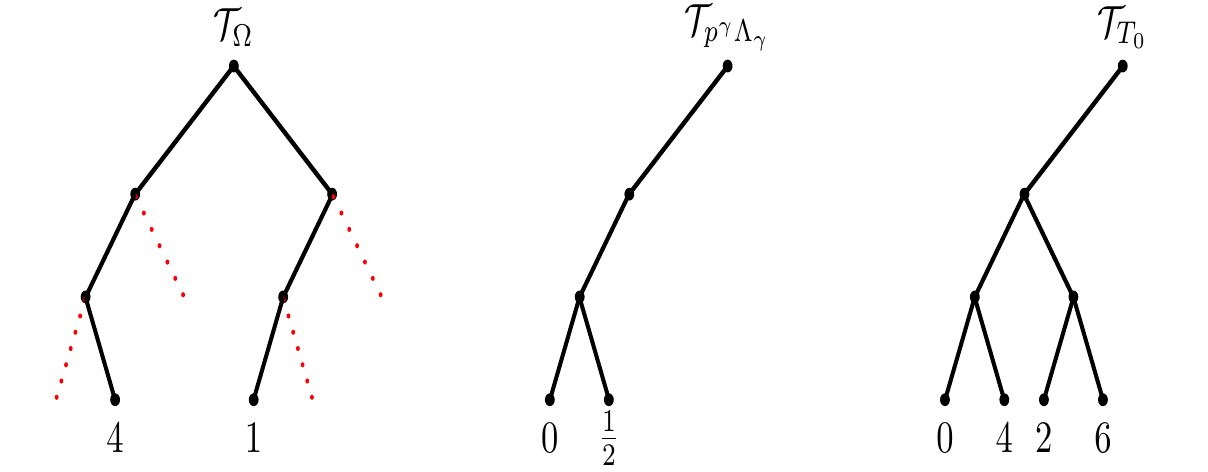}\\
  \caption{The left  is the  tree determined by $\Omega=(1+8\Z_2)\cup (4+8\Z_2)$; the middle is the tree determined by $\Lambda_{\gamma}=\{0,1/2\}$;
the right is the tree determined by  $T_{0}=\{0,2,4,6\}$.}
\label{translate}
\end{figure}

 \section{Finite spectral sets in $\mathbb{Q}_p$}

The following theorem  characterizes the  uniform probability measures supported  on  some   finite   sets  $C\subset  \mathbb{Q}_p$ and it  gives more information than Theorem \ref{discrete}.  As we shall see, 
the measure $\delta_{C}$ is a spectral measure if and only if  %for any integer $k$, %$\sharp  (C \mod p^k) = p^{i_k}$
%for some integer $i_k$, and also iff
 $C$ is represented by an infinite $p$-homegeneous tree for which,
from some level on, each parent has only one son.
Recall that 
$$  \gamma_C =  \max_{\substack {c,c^{\prime}\in C\\c\neq c^{\prime}} }v_p(c- c^{\prime}).$$

%\begin{lemma}\label{5.1} Suppose that $C$ is a finite set in $\Q_p$ and $\lambda\not=0$ such that
%      $\sum_{c\in C} \chi( \lambda c) = 0.$
%Then $v_p(\lambda) = \gamma_C +1$.
%\end{lemma}
%
%\begin{proof} It is still a  consequence of Schoenberg Lemma. Let 
%$v_p(\lambda)=-n$. Without loss of generality,
%we assume $\lambda = p^{-n}.$ Then the condition
%write
%     $\sum_{c\in C} e^{2\pi i p^{-n} (c \mod p^n)} = 0.$
%Thus Schoenberg Lemma applies.
%\end{proof}

\begin{theorem}
	The following  are equivalent.
	\begin{itemize}
		\item[{\rm (1)}] The measure $\delta_{C} $ is a spectral measure.
		\item[{\rm (2)}] For each  integer $\gamma>\gamma_C$,   $\Omega_{\gamma}:=\bigsqcup_{c\in C} B(c,p^{-\gamma})$ is a spectral set.
		\item[{\rm (3)}] For some integer $\gamma_0>\gamma_C$,   $\Omega_{\gamma_0}:=\bigsqcup_{c\in C} B(c,p^{-\gamma_0})$ is a spectral set.
		\item[{\rm (4)}] For any  integer $\gamma\in \mathbb{Z}$,   $\Omega_{\gamma}:=\bigcup_{c\in C} B(c,p^{-\gamma})$ is a spectral set.
	\end{itemize}
\end{theorem}

\begin{proof}
Without loss of generality, we assume that $C\subset \mathbb{Z}_p$, so that $\gamma_C\ge 0$.
Recall that for any integer $\gamma \in \mathbb{Z}$, $\mathbb{L}_{\gamma}$ denotes a complete set of representatives of the cosets of the subgroup $B(0,p^{
	\gamma})$ of $\mathbb{Q}_p$.
%\textcolor{blue}{The notation is not good. $\mathbb{L}_{\gamma}$ here is $\mathbb{L}_{-\gamma}$
%in the paper. We have to unify the notation.}

$(1)\Rightarrow(2)$: Fix $\gamma >\gamma_{C}$. Assume  that $\Lambda_{C}$ is a  spectrum of $\delta_{C}$, which means by Lemma 2.2 that
 \begin{align}\label{EQdiscrete11}
\forall \xi \in \mathbb{Q}_p, \quad\sum_{\lambda\in \Lambda_{C}}\left| \sum_{c\in C} \chi(-c(\lambda-\xi))\right|^2= (\sharp C)^2.
\end{align}

  Observe that we can assume $\Lambda_{C}\subset B(0,p^{\gamma})$.  We  assume $0 \in \Lambda_C$. 
  Let $\lambda$ be an
arbitrary point in $\Lambda_C,$ different from $0$. The orthogonality of $\chi_0$ and $\chi_\lambda$ is nothing but
       $$ \sum_{c\in C} \chi( \lambda c) = 0.$$
  Apply Lemma \ref{root} , we get that $|\lambda|_{p}\leq p^{\gamma_{C}+1}$.
        We conlcude that 
    $  \Lambda_C \subset B(0, p^{\gamma_C +1})$
                         $\subset B(0, p^\gamma)$
for all $\gamma > \gamma_C. $
  % because the left hand side of the above equality depends on $\Lambda_{C}$ through the terms
  % $\chi(-(c-c')(\lambda-\xi))$ for different $c$ and $c'$ in $C$ and these terms are unchanged %when $\lambda$
   %is replaced by $\lambda + y$ with $y\in B(0, p^{-\gamma})$.
   %We denote $\{\xi\}_\gamma= \xi \mod{p^{-\gamma}}$. For any $\xi\in \mathbb{Q}_p$,  $(\delta_{C}, %\Lambda_{C}+\{\xi\}_\gamma)$ is also a spectral pair.
   
   Now we check that $(\Omega_{\gamma}, \Lambda_{C}+\mathbb{L}_{\gamma})$ is a spectral pair.
Recall that
\begin{equation}\label{rx24}
\forall \zeta\in \mathbb{Q}_p, \quad
\widehat{1_{\Omega_\gamma}}(\zeta)=p^{-\gamma}1_{B(0,p^{\gamma})}(\zeta)\sum_{c\in C}\chi(-c\zeta).
\end{equation}
Fix $\xi\in \mathbb{Q}_p$. By (\ref{rx24}), we have
$$
\sum_{\lambda\in \Lambda_{C}+\mathbb{L}_\gamma}|\widehat{1_{{\Omega}_\gamma}}(\lambda-\xi)|^2
=p^{-2\gamma}\sum_{\lambda\in \Lambda_{C}+\mathbb{L}_\gamma} 1_{B(0,p^{\gamma})}(\lambda-\xi)\left|\sum_{c\in C}\chi(-c(\lambda-\xi))\right|^2
$$
Since $\Lambda_C \subset B(0,p^{\gamma})$, we have $B(\xi,p^{\gamma})\cap (\Lambda_C + \mathbb{L}_\gamma) = \Lambda_C+ \ell_\xi$
where $\ell_\xi$ is the unique point contained in $B(\xi,p^{\gamma})\cap \mathbb{L}_\gamma$.
Thus
%By (\ref{EQdiscrete1}), we have
\begin{align*}
\sum_{\lambda\in \Lambda_{C}+\mathbb{L}_\gamma}|\widehat{1_{{\Omega}_\gamma}}(\lambda-\xi)|^2
&=p^{-2\gamma}\sum_{\lambda\in \Lambda_{C}+\ell_\xi} \left| \sum_{c\in C}\chi(-c(\lambda-\xi))\right|^2\\
&=p^{-2\gamma}(\sharp C)^2=|\Omega_\gamma|^2
\end{align*}
where the second equality is a consequence of the criterion(\ref{EQdiscrete11}) and of the fact that $(\delta_C, \Lambda_{C}+\ell_\xi)$
is also a spectral pair.
This means, by Lemma 2.2, that $(\Omega_\gamma, \Lambda_{C}+\mathbb{L}_\gamma)$ is a spectral pair.	

$(2)\Rightarrow(3)$: Obviously.

$(3)\Rightarrow(4)$: Without loss of generality, we assume that $C\subset \mathbb{Z}_p$.
If $\gamma \le 0$, $\Omega_\gamma$ is equal to $p^{-\gamma}\mathbb{Z}_p$ which is spectral.
 If $1\le \gamma \le \gamma_0$, $\Omega_\gamma$ is spectral directly by the hypothesis and Theorem 1.1. Observe that $\sharp (C_{\!\!\!\! \mod p^\gamma})
 = \sharp C$  for $\gamma >\gamma_C$. Therefore, if  $\gamma >\gamma_0 >\gamma_C$,
 $C_{\!\!\!\! \mod p^\gamma}$ is $p$-homogeneous, so that $\Omega_\gamma$ is a spectral set.

$(4)\Rightarrow(1)$:
For any $\xi \in \mathbb{Q}_p$, there exists an integer $\gamma > \gamma_C$ such that $\xi\in B(0, p^{\gamma})$. Fix this $\gamma$ depending on $\xi$.  By the hypothesis,  $\Omega_\gamma$ is a spectral set.
Assume that $\Lambda_\gamma$ is a spectrum of $\Omega_\gamma$. That is to say
$$
\forall \zeta \in \mathbb{Q}_p,\quad
\sum_{\lambda\in \Lambda_{\gamma}}| \widehat{1_{\Omega_\gamma}}(\lambda-\zeta)|^2=|\Omega_\gamma|^2.
$$
We can assume that $\Lambda_\gamma$ has the form $\Lambda_C+\mathbb{L}_\gamma$ (see Theorem \ref{spetrumandtranslate}), where $\Lambda_C\subset B(0,p^{\gamma_C+1})$. %(\textcolor{red}{ by theorem \ref{main}, we could know the structure of the spectrum of a compact open set}).
By (\ref{rx24}), we have
%Recall that
%$$\widehat{1_{\Omega_\gamma}}(\lambda-\xi)=p^{\gamma}1_{B(0,p^{-\gamma})}(\lambda-\xi)\sum_{c\in %C}\chi(-c(\lambda-\xi)).$$ Then we have
\begin{align*}
|\Omega_\gamma|^2
&=\sum_{\lambda\in \Lambda_{\gamma}}| \widehat{1_{\Omega_\gamma}}(\lambda-\xi)|^2\\
&=p^{-2\gamma}\sum_{\lambda\in \Lambda_C+\mathbb{L}_\gamma} 1_{B(0,p^{-\gamma})}(\lambda-\xi)\left|\sum_{c\in C}\chi(-c(\lambda-\xi))\right|^2\\
&=p^{-2\gamma}\sum_{\lambda\in \Lambda_C}\left|\sum_{c\in C}\chi(-c(\lambda-\xi))\right|^2.
\end{align*}
Since $|\Omega_\gamma|^2=(\sharp C)^2p^{-2\gamma}$, we get 
$$
\forall \xi \in \Q_p, \quad (\sharp C)^2=\sum_{\lambda\in \Lambda_C}|\sum_{c\in C}\chi(-c(\lambda-\xi))|^2.
$$
This means that the measure $\delta_{C} $ is a spectral measure by Lemma 2.2.
\end{proof}

\section{Singular Spectral Measures }
In this section, we shall construct  a class of singular spectral measures. Let $I, J$ be two disjoint  infinite subsets of $\N$ such that
$$I\bigsqcup J=\N.$$
For any non-negative integer $\gamma$, let $I_{\gamma}= I \cap  \{0,1, \cdots \gamma-1 \}$ and $J_{\gamma}=J \cap  \{0,1, \cdots \gamma-1 \}$. Let $C_{I_\gamma,J_{\gamma}}\subset \Z/p^{\gamma}\Z$ be $p$-homogeneous  subsets corresponding to a $\mathcal{T}_{I_\gamma,J_{\gamma}}$ form tree  as described in Section \ref{Finitehomogeneoustree}. 
Considering $C_{I_\gamma,J_{\gamma}}$ as a subset of $\Z_p$, let
$$\Omega_{\gamma}=\bigsqcup_{c \in C_{I_\gamma, J_\gamma}} \left(c + p^\gamma \mathbb{Z}_p\right), \quad \gamma=0,1,2\cdots $$
be a nested sequence of compact open sets, i.e.  $\Omega_{0}\supset \Omega_{1}\supset \Omega_{2}\supset \cdots$.
It is obvious that the measures ${\frac{1}{|\Omega_{\gamma}|}\mathfrak{m}|_{\Omega_{\gamma}}}$ weakly converge to a singular measure $\mu_{I,J}$ as $\gamma \to \infty$. The measure $\mu_{I,J}$ is supported on a $p$-homogeneous, Cantor-like set of measure $0$, and the measure of an open ball with respect to  $\mu_{I,J}$  is just the ``proportion" of this support inside the ball.
All this could be well defined; the proofs are simple exercises of measure theory, in particular, applications of the Portmanteau Theorem.
We should remark that $\mu_{I,J}$ depends not only on $I$ and $J $ but also on the choice of  $C_{I_\gamma,J_\gamma}$.
Actually, the choice  of $C_{I_\gamma,J_\gamma}$ implies that  the average Dirac measures $\delta_{C_{I_\gamma,J_\gamma}}$ also converge to $\mu_{I,J}$ as $\gamma\to \infty$.
 %Statement of Theorem: (at least in  $\mathbb{Z}_p$) A compact closen set is a spectral set if %and only
%if it is homogeneous.

\begin{theorem}Under the above assumption, $\mu_{I,J}$ is a spectral measure with  the  following set
$$\Lambda=\left\{\sum _{i\in I}b_i p^{-i-1} :   b_{i}\in \{ 0, 1, \cdots, p-1\} \right\}$$
as a spectrum.
\end{theorem}

\begin{proof}
 What we have to prove is  the equality (\ref{spectral criterion})  for  the pair $(\mu_{I,J}, \Lambda)$.
Since $\mu_{I,J}$ is the weak limit of ${\frac{1}{|\Omega_{\gamma}|}\mathfrak{m}|_{\Omega_{\gamma}}}$ as $\gamma \to \infty$,  we have
$$\forall \xi\in \Q_p, \quad \widehat{\mu_{I,J}}(\xi)=\lim_{\gamma\to \infty} \frac{1}{|\Omega_{\gamma}|} \cdot \widehat{1_{\Omega_{\gamma}}}(\xi). $$
For each integer $\gamma\geq 0$, let  $\Lambda_{\gamma}=\Lambda\cap B(0,p^\gamma)$.
 %We immediately have
%$$\sum_{\lambda\in \Lambda}|\widehat{\mu_{I,J}} (\lambda -\xi)|^2 \geq \sum_{\lambda\in \Lambda_{\gamma}}|\widehat{\mu_{I,J}} (\lambda -\xi)|^2 .$$
Note that $\Omega_{\gamma}$  is a spectral set with spectrum $\Lambda_{\gamma}+ \mathbb{L}_{\gamma}$ by Theorem  \ref{spetrumandtranslate}.
% Assuming  $\xi\in B(0,p^{\gamma})$,   as in the proof of Theorem  \ref{main},  we have showed that
Then  (\ref{equ}) gives the following equality
 $$ \forall \xi \in \Q_p, \quad \sum_{\lambda\in \Lambda}| \widehat{1_{\Omega_{\gamma}}}(\lambda -\xi)|^2=\sum_{\lambda\in \Lambda_\gamma}| \widehat{1_{\Omega_{\gamma}}} (\lambda -\xi)|^2 = |\Omega_{\gamma}|^2.$$
 By Fatou's Lemma, we get that
 \begin{align}\label{ineq}
\sum_{\lambda\in \Lambda}|\widehat{\mu_{I,J}} (\lambda -\xi)|^2\leq \lim_{\gamma\to \infty}\frac{1}{|\Omega_{\gamma}|^2}\sum_{\lambda\in \Lambda}|\widehat{1_{\Omega_{\gamma}}} (\lambda -\xi)|^2=1
\end{align}
for all $\xi\in \Q_p$.
 Now, for any positive integer $\gamma_0$, we shall show that
$$ \sum_{\lambda\in \Lambda}| \widehat{1_{\Omega_{\gamma}}}(\lambda -\xi)|^2 \geq 1_{B(0,p^{\gamma_0})}(\xi), \hbox{ for all integers $\gamma\geq \gamma_0$}.$$
 %For any given $\xi\in \Qp$, there exists an integer $\gamma_0$ large enough, such that $\xi\in B(a,p^{\gamma_0})$.
 %Since ${\frac{1}{|\Omega_{\gamma}|}\mathfrak{m}|_{\Omega_{\gamma}}}\to \mu_{I,J}$ as $\gamma \to \infty$,  we have
%$$\widehat{\mu_{I,J}}(\xi)=\lim_{\gamma\to \infty} \frac{1}{|\Omega_{\gamma}|} \cdot \widehat{1_{\Omega_{\gamma}}}(\xi).$$
%For any $\xi\in B(0, p^{\gamma})$, there exist a $\gamma$
%$$=\sum_{\lambda\in \Lambda_{\gamma_0}}|\lim_{\gamma\to \infty} \frac{1}{|\Omega_{\gamma}|} \cdot \widehat{1_{\Omega_{\gamma}}} (\lambda -\xi)|^2$$
Recall that $$|\widehat{1_{\Omega_{\gamma}}}(\lambda-\xi)|^2=p^{-2\gamma} 1_{B(0,p^\gamma)}(\lambda-\xi)\sum_{c,c^{\prime}\in C_{I,J}}\chi((c-c^{\prime})(\lambda-\xi)).$$
For any $\xi \in B(0,p^{\gamma_0}) $,  observe that $$ \forall \lambda\in  \Lambda_{\gamma_0} , \quad \chi((c-c^{\prime})(\lambda-\xi))=1\hbox{ if } |c-c^{\prime}|_p\leq p^{-\gamma_0}$$ and
 $$\sum_{\lambda\in \Lambda_{\gamma_0}}\chi((c-c^{\prime})(\lambda-\xi))=0  \hbox{ if } |c-c^{\prime}|_p> p^{-\gamma_0}.$$
 For integer $\gamma \geq \gamma_0$, by calculation, there are  $p^{2\sharp(I_\gamma\setminus I_{\gamma_0})}p^{\sharp{I_{\gamma_0}}}$  pairs $(c,c^{\prime})\in C_{I_\gamma,J_\gamma}\times C_{I_{\gamma},J_{\gamma}}$ with $|c-c^{\prime}|_p\leq p^{-\gamma_0}$.
So we get that
$$\sum_{\lambda\in \Lambda_{\gamma_0}}|\widehat{1_{\Omega_{\gamma}}} (\lambda -\xi)|^2=p^{-2(\gamma-\sharp I_\gamma)} =|\Omega_{\gamma}|^2  \quad \forall \xi\in B(0,p^{\gamma_0}).$$
Thus, we have
$$\lim_{\gamma\to \infty}\frac{1}{|\Omega_{\gamma}|^2}\sum_{\lambda\in \Lambda_{\gamma_0}}|\widehat{1_{\Omega_{\gamma}}} (\lambda -\xi)|^2=1, \quad  \forall  \xi\in B(0,p^{\gamma_0}),$$
or in other words,
$$\sum_{\lambda\in \Lambda_0}|\widehat{\mu_{I,J}} (\lambda -\xi)|^2 =1,  \quad  \forall  \xi\in B(0,p^{\gamma_0}).$$
Note that  $$1\geq \sum_{\lambda\in \Lambda}|\widehat{\mu_{I,J}} (\lambda -\xi)|^2 \geq \sum_{\lambda\in \Lambda_0}|\widehat{\mu_{I,J}} (\lambda -\xi)|^2$$ 
by  (\ref{ineq}).
Since $\gamma_0$ could  arbitrarily large,  %by the inequality (\ref{ineq}), 
we have
$$\forall \xi \in \Qp, \quad \sum_{\lambda\in \Lambda}|\widehat{\mu_{I,J}} (\lambda -\xi)|^2=1. $$
\end{proof}

Assume $I_\gamma$ and $J_\gamma$ form a partition of $\{0, 1, \cdots, \gamma -1\}$
($\gamma \ge 1$) such that $C_{I_\gamma, J_\gamma}$ is a $p$-homogeneous tree. Then let
$$
      I = \bigcup_{n=0}^\infty (n\gamma + I_\gamma), \qquad J = \bigcup_{n=0}^\infty (n\gamma + J_\gamma).
$$
The measure constructed above in this special case is a self-similar measure generated by the following iterated function system:
$$
    f_c(x) = p^\gamma x + c     \qquad (c \in C: =C_{I\gamma, J_\gamma}).
$$

Let us consider two concrete examples.
\medskip

{\em Example 1.}  Let $p=2$,  $\gamma = 3$ and $C=\{0, 3, 4, 7\}$. Then
$$
     f_0(x) = 8x, \quad f_3(x) = 8x + 3, \quad f_4(x) = 8x +4, \quad f_7(x) = 8x +7.
$$
Observe that the tree structure of $\{0, 3, 4, 7\}$ is shown as follows
$$
    0= 0\cdot 1 + 0\cdot 2 + 0 \cdot 2^2, \qquad 3= 1\cdot 1 + 1\cdot 2 + 0 \cdot 2^2$$
    $$
    4= 0\cdot 1 + 0\cdot 2 + 1 \cdot 2^2,\qquad 7= 1\cdot 1 + 1\cdot 2 + 1 \cdot 2^2.
$$
\begin{figure}[H]
  \centering
  \includegraphics[width=0.8\textwidth]{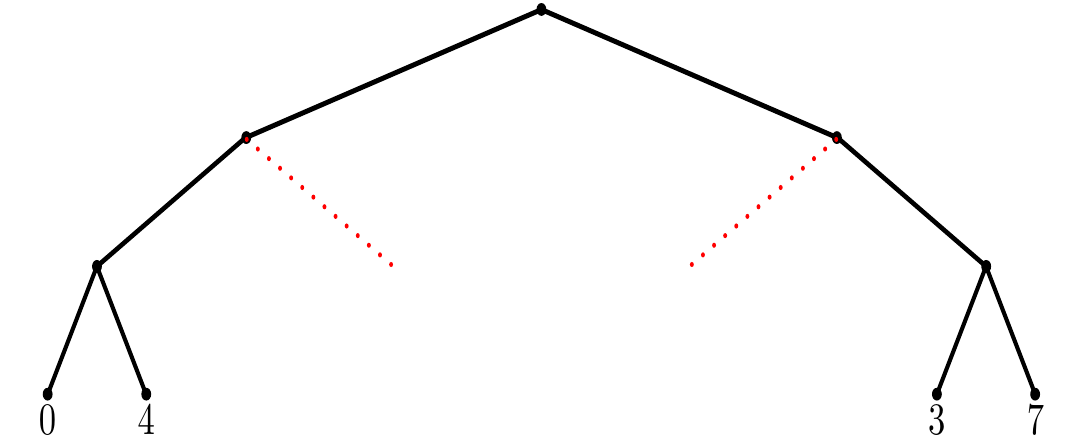}\\
  \caption{Consider $\{0,3,4,7\}$ as a  $p$-homogeneous tree.}
\label{T0347}
\end{figure}
{\em Example 2.}
 Let $p=3$,  $\gamma = 3$ and $C=\{0, 4, 8, 9, 13, 17, 18,22, 26\}$. We have
$$
    \ 0= 0\cdot 1 + 0\cdot 3 + 0 \cdot 3^2, \qquad \ 4= 1\cdot 1 + 1\cdot 3 + 0 \cdot 3^2,
              \qquad \ 8= 2\cdot 1 + 2\cdot 3 + 0 \cdot 3^2
    $$
    $$
    \ 9= 0\cdot 1 + 0\cdot 3 + 1 \cdot 3^2,\qquad 13= 1\cdot 1 + 1\cdot 3 + 1 \cdot 3^2,
                \qquad 17= 2\cdot 1 + 2\cdot 3 + 1 \cdot 3^2
$$
$$
    18= 0\cdot 1 + 0\cdot 3 + 2 \cdot 3^2,\qquad 22= 1\cdot 1 + 1\cdot 3 + 2 \cdot 3^2,
                  \qquad 26= 2\cdot 1 + 2\cdot 3 + 2 \cdot 3^2.
$$
\begin{figure}[H]
  \centering
  \includegraphics[width=1.0\textwidth]{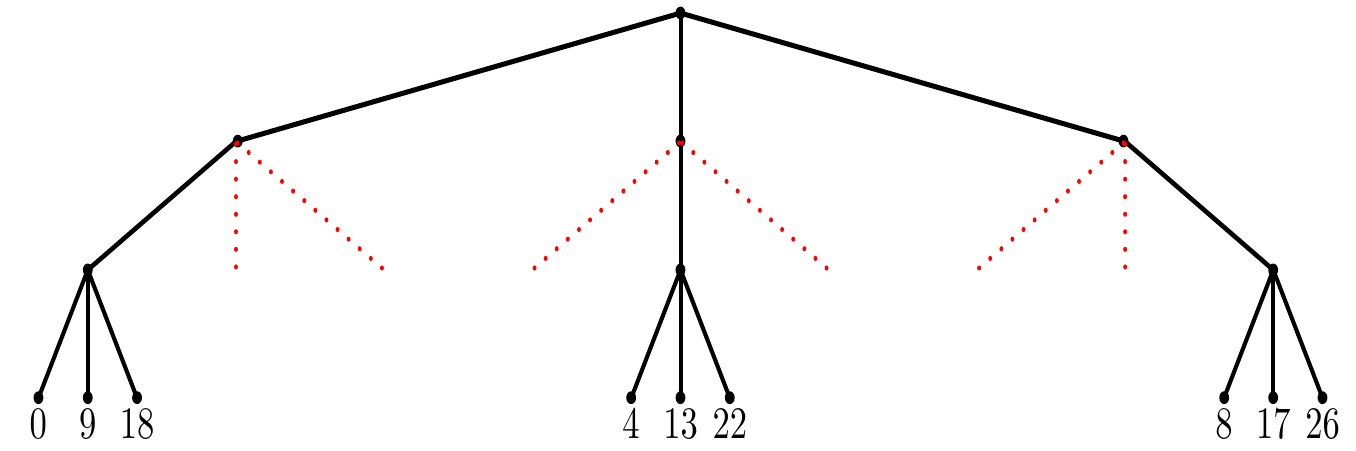}\\
  \caption{Consider $\{0, 4, 8, 9, 13, 17, 18,22, 26\}$ as a  $p$-homogeneous tree.}
\label{T0489...}
\end{figure}

\section*{acknowledgment}
The authors thank Xing-Gang He and  Lingmin Liao for their helpful discussions and  valuable remarks. The authors would also like to thank the referee for valuable suggestions and comments.
\setcounter{equation}{0}


\begin{thebibliography}{99}

%\bibitem{Cassels1986} J. W. S. Cassels,
%{\em Local Fields}, London Matnematical Society Student Texts. 3,
%Cambridge University Press, 1986.

%\bibitem{D2012}
%X. R. Dai, {\rm When does a Bernoulli convolution admit a spectrum ?}, Adv.
%Math. 231 (2012), 1681-1693.

%\bibitem{DHL2013}
 %X. R. Dai, X. G. He and C. K. Lai, {\em Spectral structure of Cantor
%measures with consecutive digits}, Adv. Math. 242 (2013), 1681-1693.

%\bibitem{DHS2009}
%D. Dutkay, D. Han and Q. Sun, {\em On spectra of a Cantor measure},
%Adv. Math. 221 (2009), 251-276.

%\bibitem{Dutkay-Han-Sun-Weber2011}
%D. Dutkay, D. Han, Q. Sun and E. Weber,{\em On the Beurling
%dimension of exponential frames}, Adv. Math. 226 (2011), 285-297.


\bibitem{BKKM10} 
D. Bose, C. P. A. Kumar, R. Krishnan and  S. Madan, 
{\em On Fuglede's conjecture for three intervals.} Online J. Anal. Comb. No. 5 (2010), 24 pp.

\bibitem{BM11}
D. Bose and  S. Madan,  {\em Spectrum is periodic for $n$-intervals.}
 J. Funct. Anal. 260 (2011), no. 1, 308-325. 

\bibitem{BM14}
D. Bose and  S. Madan,
{\em ``Spectral implies tiling'' for three intervals revisited.}
 Forum Math. 26 (2014), no. 4, 1247-1260.
 
 \bibitem{Br1953}
N. G. de Bruijn, {\em On the factorization of cyclic groups}, Indag. Math. Kon. Ned. Akad. Wretch. 15 (1953), 370-377.

\bibitem{CM1999}
 E. M. Coven and A. Meyerowitz. 
 {\em Tiling the integers with translates of one finite set.}
  J. Algebra, 212(1):161-174, 1999.


\bibitem{Fan}
A. H. Fan, {\em Spectral measures on local fields}, pp 15-25, in Difference Equations, Discrete Dynamical Systems and Applications, M. Bohner et al (eds), Springer Proceedings in Mathematics and Statistics 150. Springer, Switzerland, 2015. arXiv:1505.06230


\bibitem{FFLS2015}
A. H. Fan, S. L. Fan, L. M. Liao and R. X. Shi, {\em  Fuglede's conjecture holds in $\Q_p$}, Preprint. 	arXiv:1512.08904
%\bibitem{Daubechies}
%{ I. Daubechies}, {\em Tex lectures on wavelets},
% CBMS-NSF Regional Conference Series in Applied Mathematics, 1992.


\bibitem{FMM06}
B. Farkas, M. Matolcsi and   P. M{\'o}ra,
  {\em On Fuglede's conjecture and the existence of universal spectra. }
J. Fourier Anal. Appl. 12 (2006), no. 5, 483-494. 

%\bibitem{GM}
%{ K. Gr\"{o}chenig and W. Madych},
%    {\em Multiresolution analysis, Haar basis and self-similar tilings},
%         IEEE Trans. Inform. Theory,   38 (2) (1992), 558-568.

\bibitem{Fuglede1974}
B. Fuglede, {\em Commuting self-adjoint partial differential operators and
a group theoretic problem}, J. Funct. Anal., 16 (1974), 101-121.


 


%\bibitem{He-Lai-Lau2013}
%X. G. He, C. K. Lai and K. S. Lau, {\em Exponential spectra in $L^2(\mu)$},
%Appl. Comput. Harmon. Anal. 34 (2013), 327-338.

%\bibitem{HL2008}
%T. Y. Hu and K. S. Lau, {\em Spectral property of the Bernoulli convolutions}, Adv. Math. 219 (2008), 554-567.

\bibitem{Iosevich-Katz-Tao2003}
A. Iosevich, N. Katz  and T. Tao, 
 {\em The Fuglede spectral conjecture holds for convex planar domains}. Math. Res. Lett. 10 (2003), no. 5-6, 559-569.

\bibitem{Iosevich-Pedersen1998}
A. Iosevich and  S. Pedersen,
{\em Spectral and tiling properties of the unit cube,}
 Internat. Math. Res. Notices 16 (1998), 819-828.


\bibitem{Jorgensen-Pedersen1992}
P. Jorgensen and  S. Pedersen, 
{\em Spectral theory for Borel sets in $R^n$ of finite measure,} J. Funct. Anal. 107 (1992), 72-104. 

%\bibitem{Jorgensen-Pedersen1996}
%P. Jorgensen and S. Pedersen, {\em Harmonic analysis of fractal measures}, Const. Approx., 12 (1996), 1-30.


\bibitem{Jorgensen-Pedersen1998} P. Jorgensen and S. Pedersen,
{\em Dense analytic subspaces in fractal $L^2$-spaces}, J. Anal. Math., 75 (1998), 185-228.


\bibitem{Jorgensen-Pedersen1999} 
P. Jorgensen and S. Pedersen,
{\em Spectral pairs in Cartesian coordinates, }
J. Fourier Anal. Appl. 5 (1999), 285-302. 

\bibitem{Kolountzakis2000convex}
M. Kolountzakis,
{\em Non-symmetric convex domains have no basis of exponentials,}
 Illinois J. Math. 44 (2000), 542-550.
\bibitem{Kolountzakis2000}
 M. Kolountzakis,
 {\em Packing, tiling, orthogonality, and completeness,} 
 Bull. London Math. Soc. 32 (2000), 589-599.

\bibitem{Kolountzakis2006}
M. N. Kolountzakis  and M. Matolcsi, {\em Tiles with no spectra}, Forum Mathematicum, 18(2006), 519-528.

\bibitem{KM2006}
 M. N. Kolountzakis and M. Matolcsi, {\em Complex Hadamard matrices and the spectral set conjecture}, Collec. Math. Vol. Extra(2006), 281-291

\bibitem{Laba2001} I. Laba, 
{\em Fuglede's conjecture for a union of two intervals}, Proc. Amer.
Math. Soc. 129 (2001), 2965-2972.

\bibitem{Laba2002}
I. Laba,  {\em The spectral set conjecture and multiplicative properties of roots
of polynomials}, J. London Math. Soc.(2) 65 (2002), 661-671.

%\bibitem{Laba-Wang2002}
%I. Laba and Y. Wang, {\em On spectral Cantor measures}, J. Funct. Anal.
%193 (2002), 409-420.


\bibitem{Lagarias-Wang1996}
J. C. Lagarias and Y. Wang, {\em Integral self-affine tiles in $R^n$ I. Standard and non-standard digits sets}, J. London Math. Soc., 54 (1996), 161-179.

\bibitem{Lagarias-WangInvent}
J. C. Lagarias and Y. Wang, {\em Tiling the line with translates of one tile}, Invent. Math. 124(1996), 341-365.
\bibitem{Lagarias-Wang1997}
J. C. Lagarias and Y. Wang, {\em Spectral sets and factorizations of finite abelian groups,} J. Funct. Anal. 145 (1997), 73-98. 

\bibitem{Lagarias-Reed-Wang2000}
J. C. Lagarias, J. A. Reed and Y. Wang,
{\em Orthonormal bases of exponentials for the n-cube,}
Duke Math. J. 103 (2000), 25-37.

%\bibitem{Lai2011} C. K. Lai, {\em On Fourier frame of absolutely continuous measures}, J. Funct. Anal.
%261 (2011), 2877-2889.

%\bibitem{LLR2013}
%C. K. Lai, K.S. Lau and H. Rao, {\em Spectral structure of digit sets of
%self-similar tiles on $R^1$}, Tran. Amer. Math. Soc. 365 (2013), 3831-3850.

%\bibitem{Landau1967}
%H. Landau, {\em Necessary density conditions for sampling and interpolation of certain entire functions}, Acta Math. 117 (1967), 37-52.
\bibitem{Matolcsi2005}
M. Matolcsi, {\em Fuglede's conjecture fails in dimension 4}, Proc. Amer.
Math. Soc., 133(2005), 3021-3026.

%\bibitem{Ramakrishnan-Valenza1999}
%D. Ramakrishnan
%R. J. Valenza, {\em Fourier Analysis on
%Number Fields}, GTM 186, Springer, 1999.

\bibitem{redei1950}
 L. R\'edei. { \em Ein Beitrag zum Problem der Faktorisation von endlichen Abelschen Gruppen.} Acta Math. Acad. Sci. Hungar., 1:197-207, 1950.
 \bibitem{redei1954}
 L. R\'edei. \"Uber das Kreisteilungspolynom. Acta Math. Acad. Sci. Hungar., 5:27-28, 1954.


\bibitem{Schoenberg1964}
I. J. Schoenberg, {\em A note on the cyclotomic polynomial}, Mathematika 11 (1964) 131-136.

%\bibitem{Strichartz1998}
%R. Strichartz, {\em Remarks on Dense analytic subspace in fractal $L^2$-
%spaces}, J. Anal. Math., 75 (1998), 229-231.

%\bibitem{Strichartz2000}
 %R. Strichartz, {\em Mock Fourier series and transforms associated with
%certain Cantor measures}, J. Anal. Math., 81 (2000), 209-238.

\bibitem{Taibleson1975}
M. H. Taibleson, {\em Fourier Analysis on Local Fields}, Princeton University Press and University of Tokyo Press, 1975.

\bibitem{Tao2004}
T. Tao, {\em Fuglede's conjecture is false in 5 and higher dimensions}, Math. Res. Lett., 11 (2004), 251-258.

\bibitem{Vladimirov-Volovich-Zelenov1994}
V. S. Vladimirov, I. V. Volovich and E. I. Zelenov, {\em $p$-adic Analysis and Mathematical Physics}, World Scientific, 1994.





\end{thebibliography}
\end{document}